\documentclass[a4paper, 11pt]{amsart}

\usepackage{amssymb}
\usepackage[shortlabels]{enumitem}
\usepackage{tikz-cd}
\usepackage{mathtools}
\usepackage[linktocpage]{hyperref}

\newtheorem{thm}{Theorem}[section]
\newtheorem{prop}[thm]{Proposition}
\newtheorem{lem}[thm]{Lemma}

\theoremstyle{definition}
\newtheorem{defn}[thm]{Definition}
\newtheorem*{conv}{Convention}

\theoremstyle{remark}
\newtheorem{rmk}[thm]{Remark}
\newtheorem{ex}[thm]{Example}

\newcommand{\N}{\mathbb{N}}
\newcommand{\Z}{\mathbb{Z}}
\newcommand{\R}{\mathbb{R}}

\newcommand{\C}{\mathsf{C}}
\newcommand{\CHaus}{\mathsf{CHaus}}
\newcommand{\cont}{\mathsf{cont}}
\newcommand{\CptifSub}{\mathsf{CptifSub}}
\newcommand{\D}{\mathsf{D}}
\newcommand{\End}{\mathsf{End}}
\newcommand{\LCHaus}{\mathsf{LCHaus}}
\newcommand{\Set}{\mathsf{Set}}
\newcommand{\SFlow}{\mathsf{SFlow}}
\newcommand{\ShiftEq}{\mathsf{ShiftEq}}
\newcommand{\Sub}{\mathsf{Sub}}
\newcommand{\Sz}{\mathsf{Sz}}

\newcommand{\ds}{\displaystyle}
\newcommand{\partialto}{\rightharpoonup}
\newcommand{\xpartialto}{\xrightharpoonup}

\DeclareMathOperator{\Conl}{Conl}
\DeclareMathOperator{\Dom}{Dom}
\DeclareMathOperator{\id}{id}

\begin{document}

\title[Conley index theory without index pairs. I]
{Conley index theory without index pairs. I: The point-set level theory}
\author{Yosuke Morita}
\address{Faculty of Mathematics, Kyushu University, 
744 Motooka, Nishi-ku, Fukuoka 819-0395, Japan}
\email{y-morita@math.kyushu-u.ac.jp}

\begin{abstract}
We propose a new framework for Conley index theory. 
The main feature of our approach is that we do not use the notion of index pairs. 
We introduce, instead, the notions of compactifiable subsets and index neighbourhoods, 
and formulate and prove basic results in Conley index theory using these notions. 
We treat both the discrete time case and the continuous time case. 
\end{abstract}

\maketitle

\tableofcontents

\section{Introduction}

Conley index theory, initiated by Conley~\cite{Con78}, 
is an important tool for the study of topological dynamical systems. 
In this series of papers, we propose a new framework for this theory. 
We do not assume any prior knowledge of Conley index theory. 

In this introduction, we explain only the discrete time case for simplicity, 
although we treat both the discrete time case and the continuous time case in this paper. 

First, let us briefly summarize the traditional construction of the Conley index. 
Suppose that one is given a topological dynamical system, say, 
a self-homeomorphism $f \colon X \to X$ on a locally compact metrizable space $X$ 
(in this paper, we work in a more general setting, namely, 
a continuous partial self-map on a locally compact Hausdorff space). 
Let $S$ be an isolated $f$-invariant subset of $X$. 
In the traditional Conley index theory, 
one first defines the notion of \emph{index pairs} for $S$, 
whose precise definition varies among researchers 
(cf.\ Conley~\cite[Ch.~III, Defn.~4.1]{Con78}, Salamon~\cite[Defns.~4.1 and 5.1]{Sal85}, 
Robbin--Salamon~\cite[Defn.~5.1]{RS88}, Mrozek~\cite[Defn.~2.1]{Mro90}, 
Szymczak~\cite[Defn.~3.4]{Szy95}, Franks--Richeson~\cite[Defn.~3.1]{FR00}). 
Among them, the most general one is Robbin--Salamon's~\cite[Defn.~5.1]{RS88}. 
According to their definition, 
an index pair for $S$ is a pair $(N, L)$ satisfying the following three conditions: 
\begin{itemize}
\item $N$ and $L$ are compact subsets of $X$ satisfying $L \subset N$. 
\item $N \smallsetminus L$ is an isolating neighbourhood of $S$. 
\item The self-map $f_{(N, L)} \colon N/L \to N/L$ defined by 
\[
f_{(N, L)}(x) = 
\begin{cases}
f(x) & \text{(if $x \in (N \smallsetminus L) \cap f^{-1}(N \smallsetminus L)$)}, \\ 
\ast & \text{(otherwise)}
\end{cases}
\]
is continuous. 
\end{itemize}
One proves, then, the following results: 
\begin{itemize}
\item Every isolated $f$-invariant subset $S$ of $X$ admits an index pair. 
\item For any two index pairs $(N, L)$ and $(N', L')$ for $S$, 
one can canonically construct a family of continuous maps from $N/L$ to $N'/L'$ in a coherent way. 
\end{itemize}
The \emph{Conley index} of $S$ relative to $f$ is 
the image of the above based continuous self-map $f_{(N,L)} \colon N/L \to N/L$, 
seen as an object of the endomorphism category (cf.\ Definition~\ref{defn:disc-szymczak}~(1)), 
under an appropriate functor $Q$. 
Here, the word `appropriate' means that, roughly speaking, 
the following two conditions are satisfied: 
\begin{itemize}
\item The Conley index is independent of the choice of the index pair. 
More precisely, one can canonically construct, from the above family of continuous maps, 
a single isomorphism from $Q(f_{(N, L)})$ to $Q(f_{(N', L')})$. 
\item The Conley index is invariant under continuation of $f$. 
\end{itemize}
Various functors $Q$ satisfying these conditions are used
(cf.\ Robbin--Salamon~\cite[Defn.~9.1]{RS88}, 
Mrozek~\cite[Defn.~2.8]{Mro90}, \cite[Defn.~1.3]{Mro94}, 
Szymczak~\cite[Defn.~4.1]{Szy95}, Franks--Richeson~\cite[Defn.~4.8]{FR00}). 
Among them, the most general one is Szymczak's~\cite[Defn.~4.1]{Szy95}.

Next, let us explain our new approach. 
Our key observation is that, in the above definition of index pairs by Robbin--Salamon, 
$N$ and $L$ themselves are not important; 
we only need the information on the difference set $N \smallsetminus L$. 
To see this, let us introduce the following two notions: 
\begin{itemize}
\item We say that a locally compact Hausdorff subset $E$ of $X$ is \emph{$f$-compactifiable} if 
the self-map $f_E^+ \colon E^+ \to E^+$ defined by 
\[
f_E^+(x) =
\begin{cases}
f(x) & \text{(if $x \in E \cap f^{-1}(E)$)}, \\
\ast & \text{(otherwise)}
\end{cases}
\]
is continuous, where $E^+$ is the one-point compactification of $E$ 
(cf.\ Definition~\ref{defn:disc-compactifiable}). 
\item We say that a neighbourhood $E$ of $S$ is an \emph{index neighbourhood}
if it is isolating and $f$-compactifiable (cf.\ Definition~\ref{defn:disc-index-neighbourhood}).
\end{itemize}
Then, a pair $(N, L)$ of compact subsets of $X$ satisfying $L \subset N$ is an index pair for $S$ 
if and only if $N \smallsetminus L$ is an index neighbourhood of $S$. 
Also, if $E$ is an index neighbourhood of $S$, 
then $(\overline{E}, \overline{E} \smallsetminus E)$ is an index pair for $S$, 
where $\overline{E}$ is the closure of $E$. 
This observation suggests the possibility that the notion of index pairs is a \emph{red herring}, 
and one can develop Conley index theory in a more efficient and more transparent way 
by using the notion of index neighbourhoods instead. 
The purpose of this series of papers is to convince the readers that it is indeed the case. 

In fact, we can go further. 
One sees that crucially used in the construction of the family of continuous maps mentioned above 
is the following property of index neighbourhoods 
(cf.\ Salamon~\cite[Lem.~4.7]{Sal85} and Theorem~\ref{thm:disc-isolating-1}): 
\begin{itemize}
\item Let $E$ and $E'$ be two index neighbourhoods of $S$. 
Then, there exist $a,b,a',b' \in \N$ such that 
\[
\bigcap_{i=0}^{a+b} f^{-i}(E) \subset f^{-a}(E'), \qquad 
\bigcap_{i'=0}^{a'+b'} f^{-i'}(E') \subset f^{-a'}(E). 
\]
\end{itemize}
Let us write $E \sim_f E'$ when this condition is satisfied (cf.\ Definition~\ref{defn:disc-sim}). 
Then, one notices that the relation $\sim_f$ makes sense 
not only for isolating neighbourhoods of $S$ but more generally for all subsets of $X$. 
It is possible to construct a canonical family of based continuous maps from $E^+$ to $E'^+$ 
for any two $f$-compactifiable subsets $E$ and $E'$ satisfying $E \sim_f E'$
(cf.\ Definition~\ref{defn:disc-f_E'E} and Theorem~\ref{thm:disc-conley-2}). 
Thus, one might be able to say that 
the notions of isolated $f$-invariant subsets and index neighbourhoods are 
still somewhat superficial, and the truly important notions in Conley index theory are 
$f$-compactifiable subsets and the relation $\sim_f$ (this position is perhaps too extreme). 

Rewriting the proofs in Conley index theory using $f$-compactifiable subsets 
instead of index pairs is not a trivial task (at least to the author). 
One problem is that, in the paper of Robbin--Salamon, 
they rephrased the continuity of the self-map $N/L \to N/L$ 
into a certain condition which refers $N$ and $L$, 
not the difference set $N \smallsetminus L$ (\cite[Thms.~4.2 and 4.3]{RS88}), 
and develop the theory using this rephrased condition. 
Also, their construction of index pairs uses the Lyapunov function 
(\cite[Thms.~5.3 and 5.4]{RS88}), which is not available in our more general setting. 
Mrozek~\cite[\S 5]{Mro90}, \cite[\S 4]{Mro94} 
gave another construction of index pairs that does not use the Lyapunov function, 
but his argument seems to be difficult to translate into the language of $f$-compactifiable subsets. 
Many proofs that we give, including the proof of the existence of index neighbourhoods, 
are based on the following elementary observation: 
\begin{itemize}
\item Let $f \colon X \partialto Y$ 
be a continuous partial map of locally compact Hausdorff spaces. 
Then, the map $f^+ \colon X^+ \to Y^+$ between the one-point compactifications defined by 
\[
f^+(x) = \begin{cases} f(x) & \text{(if $x \in \Dom f$)}, \\ \ast & \text{(otherwise)} \end{cases} 
\]
is continuous if and only if $f$ is proper and openly defined 
(cf.\ Definitions~\ref{defn:disc-proper-openly-defined} and 
\ref{defn:disc-one-point-compactification}). 
\end{itemize}
This observation enables us to separate the continuity of $f^+$ into two conditions, 
the properness and the openly definedness of $f$. 
The author hopes that the proofs that we give are more transparent than the existing proofs. 

In this first paper of the series, 
we define the Conley index as an object at the point-set topology level, 
not at the homotopy category level, and study its point-set level properties. 
Homotopical properties of the Conley index, including the homotopy invariance under continuation, 
will be studied in the forthcoming paper~\cite{Mor}. 

\begin{rmk}
There is another approach to developing Conley index theory without using index pairs 
by S\'anchez-Gabites~\cite{San11}. 
\end{rmk}

Now, let us speculate some possibilities to apply our new framework to Floer theory of three-manifolds, without much supporting evidence. 

Using Conley index theory, Manolescu~\cite{Man03} constructed the Seiberg--Witten Floer 
stable homotopy type, which has more information than the usual monopole Floer homology. 
It is expected that there is an extension of the monopole Floer homology 
which assigns an $A_\infty$-category with a distinguished object to each compact three-manifold with boundary, 
and it can be used for formulating a gluing formula for the monopole Floer homology. 
One can envision that there is also an extension of the Seiberg--Witten Floer stable homotopy type 
which assigns a stable $(\infty, 1)$-category with a distinguished object to each compact three-manifold with boundary
(very roughly speaking, a stable $(\infty, 1)$-category is 
something like a stable homotopy category equipped with all the information of homotopy coherence, 
and can be seen as a nonlinear version of the $A_\infty$-category). 
However, the Conley index is usually defined at the level of the plain homotopy category, 
and because of that, Manolescu's Seiberg--Witten Floer stable homotopy type is 
defined as an object in a variant of the plain stable homotopy category. 
In this paper, we construct the Conley index purely at the point-set level so that, 
in the forthcoming paper~\cite{Mor}, 
the homotopy invariance under continuation can be established in a homotopy coherent manner. 
This could be a first step toward defining the above-mentioned extension of 
the Seiberg--Witten Floer stable homotopy type for compact three-manifolds with boundary. 

Unlike the Seiberg--Witten case, it is still not known if there is a homotopy refinement of 
the instanton Floer homology for closed three-manifolds. 
In the Seiberg--Witten case, the finite-dimensional approximation of the formal gradient flow 
satisfies a certain compactness condition (\cite[Prop.~3]{Man03}) 
that ensures the existence of an isolating neighbourhood. 
This is crucial for defining the Floer stable homotopy type via Conley index theory, 
because the Conley index is traditionally defined for isolated invariant subsets. 
An analogous compactness condition fails in the instanton case, and it seems very difficult 
to define the instanton Floer stable homotopy type by using the traditional Conley index theory. 
However, our new formulation does not necessarily require the existence of isolating neighbourhood; 
we only need a compactifiable subset. 
Thus, we can imagine the possibility that our framework works also for the instanton case, 
even though, at this moment, there is no circumstantial evidence for this speculation. 

\subsection*{Outline of the paper}
In the first half of this paper, we treat the discrete time case. 
In Section~\ref{sect:disc-Szymczak}, we give the definition of the Szymczak category, 
which is the universal receiver of the Conley index. 
In Section~\ref{sect:disc-set}, 
we give basic definitions and results on partial maps on sets, 
and define a functor $\Conl_f$ for partial self-maps on sets. 
In Section~\ref{sect:disc-topological-space}, 
we study proper and openly defined continuous partial maps, 
and refine $\Conl_f$ for continuous partial self-maps on locally compact Hausdorff spaces 
using the notion of $f$-compactifiable subsets. 
In Section~\ref{sect:disc-isolated}, 
we define the Conley index of isolated $f$-invariant sets using the notion of index neighbourhoods. 

In the second half of this paper, we treat the continuous time case. 
Many definitions, theorems, lemmas, and proofs are parallel to the discrete time case, 
while some definitions and proofs are more involved. 
We omit proofs when they are completely the same as the discrete time case. 
In Section~\ref{sect:cont-Szymczak}, 
we define an analogue of the Szymczak category in the continuous time setting. 
In Section~\ref{sect:cont-set}, 
we give basic definitions and results on partial semiflows on sets, 
and define a functor $\Conl_F$ for them. 
In Section~\ref{sect:cont-topological-space}, 
we study finite-time proper and openly defined continuous partial semiflows, 
and refine $\Conl_F$ for continuous partial semiflows on locally compact Hausdorff spaces
using the notion of $F$-compactifiable subsets. 
In Section~\ref{sect:cont-lemmas}, 
we prepare some lemmas needed in the next section. 
In Section~\ref{sect:cont-isolated}, 
we define the Conley index of isolated $F$-invariant sets using the notion of index neighbourhoods. 

In Appendix~\ref{append:proper}, we prove some results on proper maps that are needed in this paper. 
In Appendix~\ref{append:shift-equivalence}, 
we explain a categorical meaning of the Szymczak category using shift equivalences.
\begin{conv}
In this paper, every topological space is assumed to be compactly generated weak Hausdorff. 
Products, pullbacks, and subspaces are taken in the category of 
compactly generated weak Hausdorff spaces. 
\end{conv}

\part{The discrete time case}

\section{The Szymczak category}\label{sect:disc-Szymczak}

The Conley index is defined as an object of a certain category, which we call the Szymczak category: 

\begin{defn}\label{defn:disc-szymczak}
Let $\C$ be a category. 
\begin{enumerate}
\item We define the \emph{category of endomorphisms in $\C$} 
(or the \emph{category of $\N$-equivariant objects in $\C$}), 
for which we write $\End(\C)$, as follows: 
\begin{itemize}
\item An object of $\End(\C)$ is an endomorphism $f \colon X \to X$ in $\C$. 
\item For each two endomorphisms $f \colon X \to X$ and $g \colon Y \to Y$ in $\C$, 
the hom-set $\End(\C)(f, g)$ is given by 
\[ 
\End(\C)(f, g) = \{ \varphi \in \C(X, Y) \mid \varphi f = g \varphi \}. 
\]
\end{itemize}
\item We define the \emph{Szymczak category of $\C$}, 
for which we write $\Sz(\C)$, as follows: 
\begin{itemize}
\item An object of $\Sz(\C)$ is an endomorphism $f \colon X \to X$ in $\C$. 
\item For each two endomorphisms $f \colon X \to X$ and $g \colon Y \to Y$ in $\C$, 
the hom-set $\Sz(\C)(f, g)$ is given by
\[
\Sz(\C)(f, g) = (\End(\C)(f, g) \times \N) / \sim,
\]
where $(\varphi, k) \sim (\varphi', k')$ if and only if there exists $n \in \N$ such that 
$\varphi f^{k'+n} = \varphi' f^{k+n}$ (or equivalently, $g^{k'+n} \varphi = g^{k+n} \varphi'$). 
\item The composition of two morphisms 
$\overline{(\varphi, k)} \in \Sz(\C)(f, g)$ and $\overline{(\psi, \ell)} \in \Sz(\C)(g, h)$ 
is given by
$\overline{(\psi, \ell)} \circ \overline{(\varphi, k)} = \overline{(\psi \varphi, k + \ell)}$, 
which does not depend on the choice of the representatives. 
\end{itemize}
\end{enumerate}
\end{defn}

\begin{rmk}
The Szymczak category is introduced by Szymczak~\cite[\S 2.1]{Szy95} 
under the name `the category of objects equipped with a morphism'. 
\end{rmk}

We can easily see that an equivalence of two categories $\Phi \colon \C \to \D$ induces 
equivalences $\Phi \colon \End(\C) \to \End(\D)$ and $\Phi \colon \Sz(\C) \to \Sz(\D)$. 

\section{Partial maps of sets}\label{sect:disc-set}

\begin{defn}
Let $X$ and $Y$ be two sets. A \emph{partial map} from $X$ to $Y$ is a pair $(D, f)$, 
where $D$ is a subset of $X$, which we call the \emph{domain} of the partial map, 
and $f \colon D \to Y$ is a map. 
\end{defn}

We use the symbol $\partialto$ to signify a partial map. We write $\Dom f$ rather than $D$ 
to avoid specifying the domain of a partial map. 
We define the composite $gf \colon X \partialto Z$ of two partial maps
$f \colon X \partialto Y$ and $g \colon Y \partialto Z$ by 
\[
\Dom (gf) = f^{-1}(\Dom g), \qquad (gf)(x) = g(f(x)). 
\]

Let $\Set^\partial$ be the category of sets and partial maps. 
Let $\Set_\ast$ be the category of based sets and based maps. 
Then, we can define the following functor: 

\begin{defn}
We define a functor 
$({-})^+ \colon \Set^\partial \to \Set_\ast$ as follows: 
\begin{itemize}
\item Given a set $X$, 
we put $X^+ = X \sqcup \{ \ast \}$, with base point $\ast$. 
\item Given a partial map $f \colon X \partialto Y$ of sets, 
we define a based map $f^+ \colon X^+ \to Y^+$ by
\[
f^+(x) = \begin{cases} f(x) & \text{(if $x \in \Dom f$)}, \\ \ast & \text{(otherwise)}. \end{cases} 
\]
\end{itemize}
\end{defn}

\begin{prop}\label{prop:disc-plus-1}
$({-})^+ \colon \Set^\partial \to \Set_\ast$ is an equivalence of categories. 
\end{prop}
\begin{proof}
The following functor $U \colon \Set_\ast \to \Set^\partial$ is an inverse of $({-})^+$: 
\begin{itemize}
\item Given a based set $(X, x_0)$, 
put $U(X, x_0) = X \smallsetminus \{ x_0 \}$. 
\item Given a based set $f \colon (X, x_0) \to (Y, y_0)$, define a partial map 
$Uf \colon X \smallsetminus \{ x_0 \} \partialto Y \smallsetminus \{ y_0 \}$ by 
\[
\Dom (Uf) = f^{-1}(Y \smallsetminus \{ y_0 \}), \qquad (Uf)(x) = f(x). \qedhere
\]
\end{itemize}
\end{proof}

For $n \in \N$ and a partial self-map $f \colon X \partialto X$, we put
\[
f^n = \underbrace{f \circ \dots \circ f}_n \colon X \partialto X, 
\]
and for a subset $E$ of $X$, 
we use the notation $f^{-n}(E)$ to mean the inverse image $(f^n)^{-1}(E)$.

\begin{defn}
Let $f \colon X \partialto X$ be a partial self-map on a set $X$. 
For each subset $E$ of $X$, 
we define a partial self-map $f_E \colon E \partialto E$ by 
\[
\Dom f_E = E \cap f^{-1}(E), \qquad f_E(x) = f(x) 
\]
and call it the \emph{induced partial self-map on $E$}. 
\end{defn}

In this paper, by $f_E^n \ (n \in \N)$, we always mean $(f_E)^n$, not $(f^n)_E$. Thus,
\[
\Dom f_E^n = \bigcap_{i=0}^n f^{-i}(E), \qquad f_E^n(x) = f^n(x). 
\]

\begin{defn}\label{defn:disc-sim}
Let $f \colon X \partialto X$ be a partial self-map on a set $X$. 
Let $E$ and $E'$ be two subsets of $X$. 
\begin{enumerate}
\item We say that a triple $(a,b,c) \in \N^3$ is 
\emph{$(E, E')$-admissible relative to $f$} if 
\[
\bigcap_{i=0}^{a+b} f^{-i}(E) \subset f^{-a}(E'), \qquad 
\bigcap_{i'=0}^{b+c} f^{-i'}(E') \subset f^{-b}(E). 
\]
We write $A_f(E, E') \subset \N^3$ for the set of all $(E, E')$-admissible triples relative to $f$. 
\item We use the notation $E \sim_f E'$ to mean $A_f(E, E') \neq \varnothing$. 
\end{enumerate}
\end{defn}

\begin{rmk}\label{rmk:disc-sim-symmetric}
The relation $E \sim_f E'$ holds if and only if there exist 
$a, a', b, b' \in \N$ such that 
\[
\bigcap_{i=0}^{a+b} f^{-i}(E) \subset f^{-a}(E'), \qquad 
\bigcap_{i'=0}^{a'+b'} f^{-i'}(E') \subset f^{-a'}(E). 
\]
\end{rmk}

\begin{lem}\label{lem:disc-sim-equivalence-relation}
Let $f \colon X \partialto X$ be a partial self-map on a set $X$. 
\begin{enumerate}
\item For any subset $E$ of $X$, we have $A_f(E, E) = \N^3$. 
In particular, $(0,0,0) \in A_f(E, E)$. 
\item Let $E, E'$, and $E''$ be three subsets of $X$. 
If $(a,b,c) \in A_f(E, E')$ and $(a',b',c') \in A_f(E', E'')$, 
then $(a+a',b+b',c+c') \in A_f(E, E'')$. 
\end{enumerate}
\end{lem}

\begin{proof}
(1) Trivial. 

(2) We see that 
\begin{align*}
\bigcap_{j=0}^{a+a'+b+b'} f^{-j}(E) 
&= \bigcap_{i'=0}^{a'+b'} f^{-i'} 
\left( \bigcap_{i=0}^{a+b} f^{-i}(E) \right) \\
&\subset \bigcap_{i'=0}^{a'+b'} f^{-i'} (f^{-a}(E')) 
= f^{-a} \left( \bigcap_{i'=0}^{a'+b'} f^{-i'}(E') \right) \\
&\subset f^{-a}(f^{-a'}(E'')) 
= f^{-(a+a')}(E''). 
\end{align*}
Similarly, we have 
\[
\bigcap_{j''=0}^{b+b'+c+c'} f^{-j''}(E'') \subset f^{-(b+b')}(E). \qedhere
\]
\end{proof}

It follows from Remark~\ref{rmk:disc-sim-symmetric} and 
Lemma~\ref{lem:disc-sim-equivalence-relation} that $\sim_f$ is an equivalence relation 
on the set of all subsets of $X$. 

\begin{defn}\label{defn:disc-f_E'E}
Let $f \colon X \partialto X$ be a partial self-map on a set $X$. 
Let $E$ and $E'$ be two subsets of $X$. 
For each $(a,b,c) \in A_f(E, E')$, define a partial map 
$f_{E'E}^{(a,b,c)} \colon E \partialto E'$ by 
\[
\Dom f_{E'E}^{(a,b,c)} 
= \bigcap_{i=0}^{a+b} f^{-i}(E) \cap \bigcap_{i'=a}^{a+b+c} f^{-i'}(E'), \qquad
f_{E'E}^{(a,b,c)}(x) = f^{a+b+c}(x). 
\]
\end{defn}

\begin{thm}\label{thm:disc-conley-1}
Let $f \colon X \partialto X$ be a partial self-map on a set $X$. 
\begin{enumerate}
\item Let $E$ be a subset of $X$. 
Then, for any $(a,b,c) \in \N^3$, we have 
\[
f_{EE}^{(a,b,c)} = f_E^{a+b+c}.
\] 
In particular, $f_{EE}^{(0,0,0)} = \id_E$. 
\item Let $E, E'$, and $E''$ be three subsets of $X$. 
Then, for any $(a,b,c) \in A_f(E, E')$ and $(a',b',c') \in A_f(E', E'')$, 
we have 
\[
f_{E''E'}^{(a',b',c')} \circ f_{E'E}^{(a,b,c)} = f_{E''E}^{(a+a',b+b',c+c')}. 
\]
\item Let $E$ and $E'$ be two subsets of $X$. 
Then, for any $(a,b,c) \in A_f(E, E')$, we have 
\[
f_{E'E}^{(a,b,c)} \circ f_E = f_{E'} \circ f_{E'E}^{(a,b,c)}. 
\]
\item Let $E$ and $E'$ be two subsets of $X$. 
Then, for any $(a,b,c), (a',b',c') \in A_f(E, E')$, we have 
\[
f_{E'E}^{(a,b,c)} \circ f_E^{a'+b'+c'} = f_{E'E}^{(a',b',c')} \circ f_E^{a+b+c}. 
\]
\end{enumerate}
\end{thm}
\begin{proof}
(1) Trivial. 

(2) It suffices to verify that 
\[
\Dom \left( f_{E''E'}^{(a',b',c')} \circ f_{E'E}^{(a,b,c)} \right) 
= \Dom f_{E''E}^{(a+a',b+b',c+c')}. 
\]
By definition, 
\begin{align*}
&\Dom \left( f_{E''E'}^{(a',b',c')} \circ f_{E'E}^{(a,b,c)} \right) \\
&\qquad= \bigcap_{i=0}^{a+b} f^{-i}(E) \cap \bigcap_{i'=a}^{a+b+c} f^{-i'}(E') \\
&\qquad\qquad\qquad\qquad
\cap f^{-(a+b+c)} \left( \bigcap_{j'=0}^{a'+b'} f^{-j'}(E') 
\cap \bigcap_{j''=a'}^{a'+b'+c'} f^{-j''}(E'') \right) \\
&\qquad= \bigcap_{i=0}^{a+b} f^{-i}(E) \cap \bigcap_{i'=a}^{a+a'+b+b'+c} f^{-i'}(E')
\cap \bigcap_{i''=a+a'+b+c}^{a+a'+b+b'+c+c'} f^{-i''}(E''). 
\end{align*}
Since
\begin{align*}
\bigcap_{i'=a}^{a+a'+b+b'+c} f^{-i'}(E')
&= \bigcap_{j=a}^{a+a'+b'} f^{-j} \left( \bigcap_{j'=0}^{b+c} f^{-j'}(E') \right) \\
&\subset \bigcap_{j=a}^{a+a'+b'} f^{-j}(f^{-b}(E))
= \bigcap_{i=a+b}^{a+a'+b+b'} f^{-i}(E)
\end{align*}
and
\begin{align*}
\bigcap_{i'=a}^{a+a'+b+b'+c} f^{-i'}(E')
&= \bigcap_{j''=a}^{a+b+c} f^{-j''} \left( \bigcap_{j'=0}^{a'+b'} f^{-j'}(E') \right)\\
&\subset \bigcap_{j''=a}^{a+b+c} f^{-j''}(f^{-a'}(E'')) 
= \bigcap_{i''=a+a'}^{a+a'+b+c} f^{-i''}(E''), 
\end{align*}
we have
\begin{align*}
&\bigcap_{i'=a}^{a+a'+b+b'+c} f^{-i'}(E') \\
&\qquad\qquad = \bigcap_{i=a+b}^{a+a'+b+b'} f^{-i}(E)
\cap \bigcap_{i'=a}^{a+a'+b+b'+c} f^{-i'}(E')
\cap \bigcap_{i''=a+a'}^{a+a'+b+c} f^{-i''}(E'').
\end{align*}
Therefore, 
\begin{align*}
&\bigcap_{i=0}^{a+b} f^{-i}(E) \cap \bigcap_{i'=a}^{a+a'+b+b'+c} f^{-i'}(E')
\cap \bigcap_{i''=a+a'+b+c}^{a+a'+b+b'+c+c'} f^{-i''}(E'') \\
&\qquad = \bigcap_{i=0}^{a+a'+b+b'} f^{-i}(E) \cap \bigcap_{i'=a}^{a+a'+b+b'+c} f^{-i'}(E')
\cap \bigcap_{i''=a+a'}^{a+a'+b+b'+c+c'} f^{-i''}(E''). 
\end{align*}
On the other hand, by definition, 
\[
\Dom f_{E''E}^{(a+a',b+b',c+c')} 
= \bigcap_{i=0}^{a+a'+b+b'} f^{-i}(E) \cap \bigcap_{i''=a+a'}^{a+a'+b+b'+c+c'} f^{-i''}(E''). 
\]
Since
\begin{align*}
\bigcap_{i=0}^{a+a'+b+b'} f^{-i}(E)
&= \bigcap_{j'=0}^{a'+b'} f^{-j'} \left( \bigcap_{j=0}^{a+b} f^{-j}(E) \right) \\
&\subset \bigcap_{j'=0}^{a'+b'} f^{-j'}(f^{-a}(E')) 
= \bigcap_{i'=a}^{a+a'+b'} f^{-i'}(E'), 
\end{align*}
we have 
\[
\bigcap_{i=0}^{a+a'+b+b'} f^{-i}(E)
= \bigcap_{i=0}^{a+a'+b+b'} f^{-i}(E)
\cap \bigcap_{i'=a}^{a+a'+b'} f^{-i'}(E').
\]
Since
\begin{align*}
\bigcap_{i''=a+a'}^{a+a'+b+b'+c+c'} f^{-i''}(E'')
&= \bigcap_{j'=a+a'}^{a+a'+b+c} f^{-j'} \left( \bigcap_{j''=0}^{b'+c'} f^{-j''}(E'') \right) \\
&\subset \bigcap_{j'=a+a'}^{a+a'+b+c} f^{-j'}(f^{-b'}(E')) 
= \bigcap_{i'=a+a'+b'}^{a+a'+b+b'+c} f^{-i'}(E'), 
\end{align*}
we have
\[
\bigcap_{i''=a+a'}^{a+a'+b+b'+c+c'} f^{-i''}(E'')
= \bigcap_{i'=a+a'+b'}^{a+a'+b+b'+c} f^{-i'}(E')
\cap \bigcap_{i''=a+a'}^{a+a'+b+b'+c+c'} f^{-i''}(E'').
\]
Therefore, 
\begin{align*}
&\bigcap_{i=0}^{a+a'+b+b'} f^{-i}(E) \cap \bigcap_{i''=a+a'}^{a+a'+b+b'+c+c'} f^{-i''}(E'') \\
&\qquad = \bigcap_{i=0}^{a+a'+b+b'} f^{-i}(E) \cap \bigcap_{i'=a}^{a+a'+b+b'+c} f^{-i'}(E')
\cap \bigcap_{i''=a+a'}^{a+a'+b+b'+c+c'} f^{-i''}(E''). 
\end{align*}
This completes the proof of (2). 

(3) By (1) and (2), 
\[
f_{E'E}^{(a,b,c)} \circ f_E
= f_{E'E}^{(a,b,c)} \circ f_{EE}^{(0,0,1)} 
= f_{E'E}^{(a,b,c+1)} 
= f_{E'E'}^{(0,0,1)} \circ f_{E'E}^{(a,b,c)} 
= f_{E'} \circ f_{E'E}^{(a,b,c)}. 
\]

(4) By (1) and (2), 
\begin{align*}
f_{E'E}^{(a,b,c)} \circ f_E^{c'} 
&= f_{E'E}^{(a,b,c)} \circ f_{EE}^{(a',b',c')} 
= f_{E'E}^{(a+a',b+b',c+c')} \\
&= f_{E'E}^{(a',b',c')} \circ f_{EE}^{(a,b,c)}
= f_{E'E}^{(a',b',c')} \circ f_E^c. \qedhere
\end{align*}
\end{proof}

\begin{defn}\label{defn:disc-functor-1}
Let $f \colon X \partialto X$ be a partial self-map on a set $X$. 
\begin{enumerate}
\item We define a small category $\widetilde{\Sub}_f$ as follows: 
\begin{itemize}
\item An object of $\widetilde{\Sub}_f$ is a subset of $X$. 
\item For each two subsets $E$ and $E'$ of $X$, 
the hom-set $\widetilde{\Sub}_f(E, E')$ is given by $\widetilde{\Sub}_f(E, E') = A_f(E, E')$. 
\item The composition of two morphisms is given by addition. 
\end{itemize}
\item We define a functor $\widetilde{\Conl}_f \colon \widetilde{\Sub}_f \to \End(\Set_\ast)$ as follows: 
\begin{itemize}
\item For each subset $E$ of $X$, we put 
$\widetilde{\Conl}_f(E) = (f_E^+ \colon E^+ \to E^+)$. 
\item For each two subsets $E$ and $E'$ of $X$, we assign to $(a,b,c) \in A_f(E, E')$ 
the morphism $(f_{E'E}^{(a,b,c)})^+ \colon f_E^+ \to f_{E'}^+$.
\end{itemize}
\end{enumerate}
\end{defn}

By Lemma~\ref{lem:disc-sim-equivalence-relation}, $\widetilde{\Sub}_f$ is indeed a category. 
Let us verify that $\widetilde{\Conl}_f$ is a functor. 
First, by Proposition~\ref{prop:disc-plus-1} and Theorem~\ref{thm:disc-conley-1}~(3), 
 $(f_{E'E}^{(a,b,c)})^+$
is a morphism from $f_E^+$ to $f_{E'}^+$ in $\End(\Set_\ast)$. 
Second, we see from Theorem~\ref{thm:disc-conley-1}~(1) and (2) that 
$\widetilde{\Conl}_f$ is compatible with the identity morphisms and composition. 
Thus, $\widetilde{\Conl}_f$ is a functor. 

For many purposes, the functor $\widetilde{\Conl}_f$ has too much information. 
One can argue that the important thing is the relation $\sim_f$, and 
it is not appropriate to ask if $(a,b,c) \in A_f(E, E')$ or not for a particular $(a,b,c) \in \N^3$. 
For this reason, we introduce the following truncations of 
$\widetilde{\Sub}_f$ and $\widetilde{\Conl}_f$: 

\begin{defn}\label{defn:disc-functor-2}
Let $f \colon X \partialto X$ be a partial self-map on a set $X$. 
\begin{enumerate}
\item We define a small category $\Sub_f$ as follows: 
\begin{itemize}
\item An object of $\Sub_f$ is a subset of $X$. 
\item For each two subsets $E$ and $E'$ of $X$, 
the hom-set $\Sub_f(E, E')$ is given by 
\[
\Sub_f(E, E') = 
\begin{cases}
\ast & \text{(if $E \sim_f E'$)}, \\
\varnothing & \text{(otherwise)}. 
\end{cases}
\]
\end{itemize}
\item We define a functor $\Conl_f \colon \Sub_f \to \Sz(\Set_\ast)$ as follows: 
\begin{itemize}
\item For each subset $E$ of $X$, we put 
$\Conl_f(E) = (f_E^+ \colon E^+ \to E^+)$. 
\item To each two subsets $E$ and $E'$ of $X$ with $E \sim_f E'$, 
we assign the morphism $f_{E'E}^+ \colon f_E^+ \to f_{E'}^+$ defined by 
\[
f_{E'E}^+ = \overline{\left((f_{E'E}^{(a,b,c)})^+, a+b+c \right)}, 
\]
where $(a,b,c) \in A_f(E, E')$. 
\end{itemize}
\end{enumerate}
\end{defn}

If $(a,b,c), (a',b',c') \in A_f(E, E')$, 
it follows from Theorem~\ref{thm:disc-conley-1}~(4) that
\begin{align*}
\overline{\left((f_{E'E}^{(a,b,c)})^+, a+b+c\right)} 
&= \overline{\left((f_{E'E}^{(a,b,c)})^+ \circ (f_E^{a'+b'+c'})^+, a+b+c+a'+b'+c'\right)} \\
&= \overline{\left((f_{E'E}^{(a',b',c')})^+ \circ (f_E^{a+b+c})^+, a+b+c+a'+b'+c'\right)} \\
&= \overline{\left((f_{E'E}^{(a',b',c')})^+, a'+b'+c'\right)}.
\end{align*}
Thus, the definition of $f_{E'E}^+$ does not depend on the choice of $(a,b,c) \in A_f(E, E')$, 
and $\Conl_f$ is a well-defined functor. 

\begin{rmk}
Since $\sim_f$ is a symmetric relation, every morphism in $\Sub_f$ is an isomorphism. 
\end{rmk}

\section{Continuous partial maps of topological spaces}\label{sect:disc-topological-space}

\begin{defn}
Let $X$ and $Y$ be two topological spaces. We say that a partial map 
$f \colon X \partialto Y$ is \emph{continuous} if $f$ is continuous as a map from $\Dom f$ to $Y$. 
\end{defn}

The composite of two continuous partial maps is continuous. 

\begin{defn}\label{defn:disc-proper-openly-defined}
Let $f \colon X \partialto Y$ be a continuous partial map of topological spaces. 
\begin{enumerate}
\item We say that $f$ is \emph{proper} if $f$ is proper as a map from $\Dom f$ to $Y$. 
\item We say that $f$ is \emph{openly defined} if $\Dom f$ is open in $X$. 
\end{enumerate}
\end{defn}

The composite of two proper (resp.\ openly defined) partial maps is 
proper (resp.\ openly defined). 

We write $\LCHaus^\partial$ for the category whose objects are locally compact Hausdorff spaces and 
whose morphisms are proper and openly defined partial maps. 
Let $\CHaus_\ast$ be the category of based compact Hausdorff spaces and based continuous maps. 
The significance of proper and openly defined partial maps is 
that we can define the following functor: 

\begin{defn}\label{defn:disc-one-point-compactification}
We define a functor 
$({-})^+ \colon \LCHaus^\partial \to \CHaus_\ast$ as follows: 
\begin{itemize}
\item Given a locally compact Hausdorff space $X$, 
we put $X^+$ to be the one-point compactification of $X$, with base point $\ast$. 
\item Given a proper and openly defined partial map $f \colon X \partialto Y$ 
of locally compact Hausdorff spaces, 
we define a based continuous map $f^+ \colon X^+ \to Y^+$ by
\[
f^+(x) = \begin{cases} f(x) & \text{(if $x \in \Dom f$)}, \\ \ast & \text{(otherwise)}. \end{cases} 
\]
\end{itemize}
\end{defn}

Let us verify that $f^+$ is indeed continuous. 
Suppose that $L$ is a closed subset of $Y^+$ that does not contain $\ast$. We have 
\[
(f^+)^{-1}(L) = f^{-1}(L). 
\]
Since $L$ is a compact subset of $Y$ and $f \colon \Dom f \to Y$ is a proper map, 
$f^{-1}(L)$ is a compact subset of $X$. Thus, $(f^+)^{-1}(L)$ is closed in $X^+$. 
Suppose that $L$ is a closed subset of $Y^+$ that contains the base point $\ast$. We have 
\[
(f^+)^{-1}(L) = f^{-1}(L \smallsetminus \{ \ast \}) \cup (X \smallsetminus \Dom f) \cup \{ \ast \}.
\]
Since $L \smallsetminus \{ \ast \}$ is a closed subset of $Y$ and $\Dom f$ is an open subset of $X$, 
$f^{-1}(L \smallsetminus \{ \ast \}) \cup (X \smallsetminus \Dom f)$ is closed in $X$. 
Thus, $(f^+)^{-1}(L)$ is closed in $X^+$. 

\begin{prop}\label{prop:disc-plus-2}
$({-})^+ \colon \LCHaus^\partial \to \CHaus_\ast$ is an equivalence of categories. 
\end{prop}
\begin{proof}
The following functor $U \colon \CHaus_\ast \to \LCHaus^\partial$ 
is an inverse of $({-})^+$: 
\begin{itemize}
\item Given a based compact Hausdorff space $(X, x_0)$, 
put $U(X, x_0) = X \smallsetminus \{ x_0 \}$. 
\item Given a based continuous map $f \colon (X, x_0) \to (Y, y_0)$, define a partial map 
$Uf \colon X \smallsetminus \{ x_0 \} \partialto Y \smallsetminus \{ y_0 \}$ by 
\[
\Dom (Uf) = f^{-1}(Y \smallsetminus \{ y_0 \}), \qquad (Uf)(x) = f(x).
\]
\end{itemize}
For the sake of completeness, 
let us verify that $Uf$ is indeed proper and openly defined. 
Consider the following pullback diagram: 
\[
\begin{tikzcd}
\Dom (Uf) \arrow[r, "Uf"] \arrow[d, hook] & Y \smallsetminus \{ y_0 \} \arrow[d, hook] \\
X \arrow[r, "f"] & Y.
\end{tikzcd}
\]
The bottom map is proper since $X$ is compact Hausdorff (Lemma~\ref{lem:proper-source-compact}). 
Thus, the top map is also proper (Lemma~\ref{lem:proper-base-change}). 
Since $Y \smallsetminus \{ y_0 \}$ is open in $X$, 
the inverse image $f^{-1}(Y \smallsetminus \{ y_0 \}) \ (= \Dom (Uf))$ 
is open in $X$. 
Since $\Dom (Uf)$ is contained in $X \smallsetminus \{ x_0 \}$, 
it is also open in $X \smallsetminus \{ x_0 \}$. 
\end{proof}

\begin{thm}\label{thm:disc-conley-2}
Let $f \colon X \to X$ be a continuous partial self-map on a topological space $X$. 
Let $E$ and $E'$ be two subsets of $X$. Let $(a,b,c) \in A_f(E, E')$. 
\begin{enumerate}
\item If $f_E$ and $f_{E'}$ are proper, $f_{E'E}^{(a,b,c)}$ is also proper. 
\item If $f_E$ and $f_{E'}$ are openly defined, $f_{E'E}^{(a,b,c)}$ is also openly defined. 
\end{enumerate}
\end{thm}
\begin{proof}
(1) The map $f_{E'E}^{(a,b,c)} \colon \Dom f_{E'E}^{(a,b,c)} \to E'$ 
is factorized as follows: 
\begin{align*}
\Dom f_{E'E}^{(a,b,c)} 
&= \bigcap_{i=0}^{a+b} f^{-i}(E) \cap \bigcap_{i'=a}^{a+b+c} f^{-i'}(E') \\
&\qquad\xrightarrow{f_E^a} 
\bigcap_{i=0}^b f^{-i}(E) \cap \bigcap_{i'=0}^{b+c} f^{-i'}(E') 
\hookrightarrow \bigcap_{i'=0}^{b+c} f^{-i'}(E')
\xrightarrow{f_{E'}^{b+c}} E'. 
\end{align*}
Let us prove that the three maps in this factorization are all proper. 

Since $f_{E'} \colon E' \partialto E'$ is proper, the third map $f_{E'}^{b+c}$ is proper.

The first map $f_E^a$ fits into the following pullback diagram: 
\[
\begin{tikzcd}
\ds \bigcap_{i=0}^{a+b} f^{-i}(E) \cap \bigcap_{i'=a}^{a+b+c} f^{-i'}(E') \arrow[r, "f_E^a"] \arrow[d, hook] & 
\ds \bigcap_{i=0}^b f^{-i}(E) \cap \bigcap_{i'=0}^{b+c} f^{-i'}(E') \arrow[d, hook] \\
\ds \bigcap_{i=0}^a f^{-i}(E) \arrow[r, "f_E^a"] & E.
\end{tikzcd}
\]
Since $f_E \colon E \partialto E$ is proper, 
the bottom map is proper. Hence the top map is also proper (Lemma~\ref{lem:proper-base-change}). 

The second map fits into the following pullback diagram: 
\[
\begin{tikzcd}
\ds \bigcap_{i=0}^b f^{-i}(E) \cap \bigcap_{i'=0}^{b+c} f^{-i'}(E') \arrow[r, hook] \arrow[d, hook] & 
\ds \bigcap_{i'=0}^{b+c} f^{-i'}(E') \arrow[d, hook] \\
\ds \bigcap_{i=0}^b f^{-i}(E) \arrow[r, hook] & 
f^{-b}(E).
\end{tikzcd}
\]
To see that the top map is proper, it suffices to see that the bottom map is proper (Lemma~\ref{lem:proper-base-change}). 
Consider the following commutative diagram: 
\[
\begin{tikzcd}
\ds \bigcap_{i=0}^b f^{-i}(E) \arrow[r, hook] \arrow[dr, "f_E^b \quad" below] & 
f^{-b}(E) \arrow[d, "f^b"] \\
& E. 
\end{tikzcd}
\]
Since $f_E \colon E \partialto E$ is proper, the diagonal map is proper. 
By Lemma~\ref{lem:proper-composition-cancellation}~(2), the top map is also proper. 

(2) The inclusion map $\Dom f_{E'E}^{(a,b,c)} \hookrightarrow E$ is factorized as follows: 
\begin{align*}
\Dom f_{E'E}^{(a,b,c)} 
= \bigcap_{i=0}^{a+b} f^{-i}(E) \cap \bigcap_{i'=a}^{a+b+c} f^{-i'}(E')
\hookrightarrow \bigcap_{i=0}^{a+b} f^{-i}(E)
\hookrightarrow E.
\end{align*}
Let us prove that the two maps in this factorization are open inclusions. 

Since $f_E \colon E \partialto E$ is openly defined, the second map is an open inclusion.

Let us prove that the first map is an open inclusion. 
Consider the following pullback diagram: 
\[
\begin{tikzcd}
\ds \bigcap_{i=0}^{a+b} f^{-i}(E) \cap \bigcap_{i'=a}^{a+b+c} f^{-i'}(E') \arrow[r, hook] \arrow[d, hook] & 
\ds \bigcap_{i=0}^{a+b} f^{-i}(E) \arrow[d, hook] \\
\ds \bigcap_{i'=a}^{a+b+c} f^{-i'}(E') \arrow[r, hook] \arrow[d, "f^a" left] & 
f^{-a}(E') \arrow[d, "f^a"] \\
\ds \bigcap_{i'=0}^{b+c} f^{-i'}(E') \arrow[r, hook] &
E'.
\end{tikzcd}
\]
Since $f_{E'} \colon E' \to E'$ is openly defined, 
the bottom map is an open inclusion. 
Thus, the top map is also an open inclusion. 
\end{proof}

\begin{defn}\label{defn:disc-compactifiable}
Let $f \colon X \partialto X$ be a continuous partial self-map 
on a locally compact Hausdorff space $X$. 
We say that a locally closed subset $E$ of $X$ is \emph{$f$-compactifiable} 
if the induced partial self-map $f_E \colon E \partialto E$ is proper and openly defined.
\end{defn}

\begin{rmk}
In the above definition, $f$-compactifiable subsets are assumed to be locally closed in $X$ 
so that they are locally compact Hausdorff (cf.\ \cite[Ch.~I, \S 9, Props.~12 and 13]{Bou71}). 
\end{rmk}

Let $\widetilde{\CptifSub}_f$ (resp.\ $\CptifSub_f$) be the full category of $\widetilde{\Sub}_f$ 
(resp.\ $\Sub_f$) consisting of $f$-compactifiable subsets of $X$. 
By Proposition~\ref{prop:disc-plus-2} and Theorem~\ref{thm:disc-conley-2}, 
the functor $\widetilde{\Conl}_f \colon \widetilde{\Sub}_f \to \End(\Set_\ast)$ 
(resp.\ $\Conl_f \colon \Sub_f \to \Sz(\Set_\ast)$) induces a functor from 
$\widetilde{\CptifSub}_f$ to $\End(\CHaus_\ast)$ (resp.\ from $\CptifSub_f$ to $\Sz(\CHaus_\ast)$). 

\begin{defn}
Let $f \colon X \partialto X$ be a continuous partial self-map 
on a locally compact Hausdorff space $X$. 
We use the same notations $\widetilde{\Conl}_f \colon \widetilde{\CptifSub}_f \to \End(\CHaus_\ast)$ 
and $\Conl_f \colon \CptifSub_f \to \Sz(\CHaus_\ast)$ for the functors defined above. 
\end{defn}

\section{The Conley index of isolated $f$-invariant subsets}\label{sect:disc-isolated}

\begin{defn}
Let $f \colon X \partialto X$ be a partial self-map on a set $X$. 
We say that a subset $S$ of $X$ is \emph{$f$-invariant} if $S \subset \Dom f$ and $f(S) = S$.
\end{defn}

\begin{defn}
Let $f \colon X \partialto X$ be a partial self-map on a set $X$. 
Let $E$ be a subset of $X$. 
We define the \emph{$f$-invariant part} $I_f(E)$ of $E$ by 
\[
I_f(E) = \bigcap_{a,b \in \N} f^a \left( \bigcap_{i=0}^{a+b} f^{-i}(E) \right).
\]
\end{defn}

\begin{rmk}\label{rmk:disc-decreasing}
The family 
\[
\left( f^a \left( \bigcap_{i=0}^{a+b} f^{-i}(E) \right) \right)_{a,b \in \N}
\]
is decreasing in the following sense: 
for any $a,b,a',b' \in \N$ satisfying $a \leqslant a'$ and $b \leqslant b'$, we have
\[
f^{a'}\left( \bigcap_{i=0}^{a'+b'} f^{-i}(E) \right) 
\subset f^a\left( \bigcap_{i=0}^{a+b} f^{-i}(E) \right).
\]
\end{rmk}

The above definition of the $f$-invariant part agrees with the usual one, namely: 

\begin{lem}\label{lem:disc-invariant-part}\footnote{This Lemma is wrong. See Correction.}
Let $f \colon X \partialto X$ be a partial self-map on a set $X$. 
Let $E$ be a subset of $X$. Then, $I_f(E)$ is the largest $f$-invariant subset contained in $E$. 
\end{lem}

\begin{proof}
We remark that the largest $f$-invariant subset of $E$ exists. 
Indeed, if $(S_i)_{i \in I}$ is a family of $f$-invariant subsets of $E$, 
the union $\bigcup_{i \in I} S_i$ is also $f$-invariant. 

Let us prove that $I_f(E)$ is $f$-invariant. 
We have 
$I_f(E) \subset \Dom f$ and $I_f(E) \subset E$. 
By Remark~\ref{rmk:disc-decreasing}, we have
\begin{align*}
f(I_f(E)) 
&= f \left( \bigcap_{a,b \in \N} f^a\left( \bigcap_{i=0}^{a+b+1} f^{-i}(E) \right) \right) \\
&= \bigcap_{a,b \in \N} f^{a+1}\left( \bigcap_{i=0}^{a+b+1} f^{-i}(E) \right) 
= I_f(E). 
\end{align*}

Let us prove that every $f$-invariant subset of $E$ is contained in $I_f(E)$. 
It suffices to see that, if $S$ is an $f$-invariant subset of $E$, then 
\[
S \subset f^a\left( \bigcap_{i=0}^{a+b} f^{-i}(E) \right)
\]
holds for any $a,b \in \N$. 
Take any $x \in S$. 
Since $S \subset f^a(S)$, there exists $x' \in S$ such that $f^a(x') = x$. 
For each $i \in \N$ with $0 \leqslant i \leqslant a+b$, we have 
$x' \in f^{-i}(E)$ since $f^i(S) \subset S \ (\subset E)$. Thus, 
\[
x = f^a(x') \in f^a \left( \bigcap_{i=0}^{a+b} f^{-i}(E) \right). \qedhere
\]
\end{proof}

\begin{defn}\label{defn:disc-isolating}
Let $f \colon X \partialto X$ be a continuous partial self-map 
on a locally compact Hausdorff space $X$. 
Let $S$ be an $f$-invariant subset of $X$. 
\begin{enumerate}
\item A neighbourhood $E$ of $S$ is called \emph{isolating} if 
the following three conditions are satisfied: 
\begin{itemize}
\item $E$ is relatively compact. 
\item $\overline{E} \subset \Dom f$. 
\item $S = I_F(\overline{E})$.
\end{itemize} 
Here, $\overline{E}$ denotes the closure of $E$. 
\item $S$ is called \emph{isolated} if it admits an isolating neighbourhood. 
\end{enumerate}
\end{defn}

\begin{rmk}\label{rmk:disc-isolating-closure-smaller}
If $E$ is an isolating neighbourhood of $S$, 
the closure $\overline{E}$ is also isolating. 
If $E$ is an isolating neighbourhood of $S$ 
and $E'$ is a neighbourhood of $S$ contained in $E$,
then $E'$ is also isolating. 
\end{rmk}

\begin{lem}\label{lem:disc-compact}
Let $f \colon X \partialto X$ be a continuous partial self-map 
on a locally compact Hausdorff space $X$. 
Let $S$ be an isolated $f$-invariant subset of $X$. 
Then, $S$ is compact. 
\end{lem}
\begin{proof}
Take an isolating neighbourhood $K$ of $S$. 
By Remark~\ref{rmk:disc-isolating-closure-smaller}, we may assume that $K$ is compact. 
We have 
\[
S = \bigcap_{a,b \in \N} f^a \left( \bigcap_{i=0}^{a+b} f^{-i}(K) \right). 
\]
Observe that, for each $a,b \in \N$, 
\[
f^a \left( \bigcap_{i=0}^{a+b} f^{-i}(K) \right)
\]
is compact and closed in $K$. Thus, $S$ is compact. 
\end{proof}

\begin{thm}\label{thm:disc-isolating-1}
Let $f \colon X \partialto X$ be a continuous partial self-map 
on a locally compact Hausdorff space $X$. 
Let $S$ be an isolated $f$-invariant subset of $X$. 
Then, for any two isolating neighbourhoods $E$ and $E'$ of $S$, we have $E \sim_f E'$. 
\end{thm}
\begin{proof}
Let $K = \overline{E}$ be the closure of $E$ and $U = (E')^\circ$ be the interior of $E'$. 
We have 
\[
(K \smallsetminus U) 
\cap \bigcap_{a,b \in \N} f^a \left( \bigcap_{i=0}^{a+b} f^{-i}(K) \right) 
= (K \smallsetminus U) \cap S 
= \varnothing. 
\]
By the compactness of $K$ and Remark~\ref{rmk:disc-decreasing}, 
there exist $a_0, b_0 \in \N$ such that 
\[
(K \smallsetminus U) \cap f^{a_0} \left( \bigcap_{i=0}^{a_0+b_0} f^{-i}(K) \right) 
= \varnothing. 
\]
This can be rephrased as 
\[
\bigcap_{i=0}^{a_0+b_0} f^{-i}(K) \subset f^{-a_0}(U). 
\]
In particular, 
\[
\bigcap_{i=0}^{a_0+b_0} f^{-i}(E) \subset f^{-a_0}(E'). 
\]
We can similarly prove that there exist $a'_0, b'_0 \in \N$ with $a'_0 \leqslant b'_0$ such that 
\[
\bigcap_{i'=0}^{a'_0+b'_0} f^{-i'}(E') \subset f^{-a'_0}(E). \qedhere
\]
\end{proof}

\begin{defn}\label{defn:disc-index-neighbourhood}
Let $f \colon X \partialto X$ be a continuous partial self-map 
on a locally compact Hausdorff space $X$. 
Let $S$ be an $f$-invariant subset of $X$. 
We say that a neighbourhood $E$ of $S$ is an \emph{index neighbourhood} if 
it is isolating and $f$-compactifiable. 
\end{defn}

\begin{thm}\label{thm:disc-isolating-2}
Let $f \colon X \partialto X$ 
be a continuous partial self-map on a locally compact Hausdorff space $X$. 
Let $S$ be an isolated $f$-invariant subset of $X$. 
Then, the set of index neighbourhoods of $S$ forms a neighbourhood base for $S$ 
(in particular, there exists an index neighbourhood of $S$). 
\end{thm}

\begin{proof}
Take any neighbourhood $N$ of $S$. 
Since $S$ is compact by Lemma~\ref{lem:disc-compact} and $X$ is locally compact Hausdorff, 
we can take a compact neighbourhood $K$ of $S$ contained in $N$. 
Let $U = K^\circ$ the interior of $K$. 
By Remark~\ref{rmk:disc-isolating-closure-smaller}, we may assume that 
$K$ and $U$ are isolating neighbourhoods of $S$. 
By Theorem~\ref{thm:disc-isolating-1}, we have $K \sim_f U$. 
Take $(a,b,c) \in A_f(K, U)$ arbitrarily, and put 
\[
E = \bigcap_{i=0}^{a+b} f^{-i}(K) \cap \bigcap_{i'=a}^{a+b+c} f^{-i'}(U). 
\]
By Remark~\ref{rmk:disc-isolating-closure-smaller}, $E$ is an isolating neighbourhood of $S$. 
We have $E \subset N$. 
Since $K \subset \Dom f$, we see that $K \cap f^{-1}(K)$ is closed in $K$, hence compact. 
Thus, the induced partial self-map $f_K \colon K \partialto K$ is proper 
(Lemma~\ref{lem:proper-source-compact}). 
Since $U \subset \Dom f$, we see that $U \cap f^{-1}(U)$ is open in $U$, i.e.\ 
the induced partial self-map $f_U \colon U \partialto U$ is openly defined. 
Also, $E$ is locally closed in $X$. 
Now, it is enough to prove the following lemma: 

\begin{lem}\label{lem:disc-compactifiable-construction}
Let $f \colon X \partialto X$ be a continuous partial self-map on a topological space $X$. 
Let $E$ and $E'$ be two subsets of $X$ and $(a,b,c) \in A_f(E, E')$. 
Assume that $f_E \colon E \partialto E$ is proper and 
$f_{E'} \colon E' \partialto E'$ is openly defined. Put 
\[
E'' = \bigcap_{i=0}^{a+b} f^{-i}(E) \cap \bigcap_{i'=a}^{a+b+c} f^{-i'}(E').
\]
Then, the induced partial self-map $f_{E''} \colon E'' \partialto E''$ is proper and openly defined. 
\end{lem}

\begin{rmk}
In the setting of Lemma~\ref{lem:disc-compactifiable-construction}, 
$E''$ satisfies $E'' \sim_f E$ (and hence, $E'' \sim_f E'$). 
Indeed, we have $E'' \subset E$ and 
\[
\bigcap_{i=0}^{a+2b+c} f^{-i}(E) \subset E''. 
\]
\end{rmk}

\begin{proof}[Proof of Lemma~\ref{lem:disc-compactifiable-construction}]
Let us prove that $f_{E''}$ is proper. 
Consider the following commutative diagram:
\[
\begin{tikzcd}
E'' \cap f^{-1}(E'') \arrow[r, "f_{E''}"] \arrow[d, hook] & E'' \arrow[d, hook] \\
E \cap f^{-1}(E) \arrow[r, "f_E"] & E.
\end{tikzcd}
\]
Since the bottom map is proper, in order to see that the top map is proper, 
it is enough to verify that the square is a pullback, i.e.\ 
$f_E^{-1}(E'') = E'' \cap f^{-1}(E'')$ (Lemma~\ref{lem:proper-base-change}). 
We see that
\[
f_E^{-1}(E'')
= \bigcap_{i=0}^{a+b+1} f^{-i}(E) \cap \bigcap_{i'=a+1}^{a+b+c+1} f^{-i'}(E')
\]
and 
\[
E'' \cap f^{-1}(E'') = \bigcap_{i=0}^{a+b+1} f^{-i}(E) \cap \bigcap_{i'=a}^{a+b+c+1} f^{-i'}(E').
\]
Since 
\[
\bigcap_{i=0}^{a+b+1} f^{-i}(E) \subset \bigcap_{i=0}^{a+b} f^{-i}(E) \subset f^{-a}(E'), 
\]
we obtain
\[
\bigcap_{i=0}^{a+b+1} f^{-i}(E) = \bigcap_{i=0}^{a+b+1} f^{-i}(E) \cap f^{-a}(E')
\]
and
\[
\bigcap_{i=0}^{a+b+1} f^{-i}(E) \cap \bigcap_{i'=a+1}^{a+b+c+1} f^{-i'}(E')
= \bigcap_{i=0}^{a+b+1} f^{-i}(E) \cap \bigcap_{i'=a}^{a+b+c+1} f^{-i'}(E').
\]
We have thus proved $f_E^{-1}(E'') = E'' \cap f^{-1}(E'')$.

Let us prove that $f_{E''}$ is openly defined. 
Consider the following diagram: 
\[
\begin{tikzcd}
E'' \cap f^{-1}(E'') \arrow[r, hook] \arrow[d, "f^{a+b+c}" left] & E'' \arrow[d, "f^{a+b+c}"] \\
E' \cap f^{-1}(E') \arrow[r, hook] & E'.
\end{tikzcd}
\]
Since $E' \cap f^{-1}(E')$ is open in $E'$, in order to see $E'' \cap f^{-1}(E'')$ is open in $E''$, it is enough to verify that the square is a pullback, i.e.\ 
$E'' \cap f^{-(a+b+c)}(E' \cap f^{-1}(E')) = E'' \cap f^{-1}(E'')$. 
We see that 
\[
E'' \cap f^{-(a+b+c)}(E' \cap f^{-1}(E'))
= \bigcap_{i=0}^{a+b} f^{-i}(E) \cap \bigcap_{i'=a}^{a+b+c+1} f^{-i'}(E')
\]
and 
\[
E'' \cap f^{-1}(E'') = \bigcap_{i=0}^{a+b+1} f^{-i}(E) \cap \bigcap_{i'=a}^{a+b+c+1} f^{-i'}(E').
\]
Since 
\begin{align*}
\bigcap_{i'=a}^{a+b+c+1} f^{-i'}(E')
&\subset f^{-(a+1)} \left( \bigcap_{j'=0}^{b+c} f^{-j'}(E') \right) \\
&\subset f^{-(a+1)} (f^{-b}(E)) = f^{-(a+b+1)}(E), 
\end{align*}
we obtain
\[
\bigcap_{i'=a}^{a+b+c+1} f^{-i'}(E') = f^{-(a+b+1)}(E) \cap \bigcap_{i'=a}^{a+b+c+1} f^{-i'}(E')
\]
and
\[
\bigcap_{i=0}^{a+b} f^{-i}(E) \cap \bigcap_{i'=a}^{a+b+c+1} f^{-i'}(E')
= \bigcap_{i=0}^{a+b+1} f^{-i}(E) \cap \bigcap_{i'=a}^{a+b+c+1} f^{-i'}(E').
\]
We have thus proved $E'' \cap f^{-(a+b+c)}(E' \cap f^{-1}(E')) = E'' \cap f^{-1}(E'')$. 
\end{proof}

The proof of Theorem~\ref{thm:disc-isolating-2} is completed. 
\end{proof}

\begin{defn}
Let $f \colon X \partialto X$ be a continuous partial self-map 
on a locally compact Hausdorff space $X$. 
Let $S$ be an isolated $f$-invariant subset of $X$. 
We define the \emph{Conley index of $S$ relative to $f$} to be 
$\Conl_f(E)$, where $E$ is an index neighbourhood of $S$. 
\end{defn}

By Theorems~\ref{thm:disc-isolating-1} and \ref{thm:disc-isolating-2}, 
the full subcategory of $\CptifSub_f$ consisting of all index neighbourhoods of $S$
is equivalent to the category with one object and one morphism. 
Thus, the Conley index $\Conl_f(E)$ is well-defined as an object of $\Sz(\CHaus_\ast)$ 
up to unique isomorphism.

\begin{ex}\label{ex:disc-examples}
There are $\sim_f$-equivalence classes of $f$-compactifiable subsets 
that cannot be obtained from Theorems~\ref{thm:disc-isolating-1} and \ref{thm:disc-isolating-2}:
\begin{enumerate}
\item Fix $\theta \in \R/2\pi\Z$. Define a continuous map $f \colon \R^2 \to \R^2$ by 
\[
f(x,y) = (x \cos \theta - y \sin \theta, x \sin \theta + y \cos \theta).
\]
For each $r \in \R_{\geqslant 0}$, 
\[
U_r = \{ (x,y) \in \R^2 \mid x^2+y^2 < r^2 \}
\]
and
\[
K_r = \{ (x,y) \in \R^2 \mid x^2+y^2 \leqslant r^2 \}
\]
are both $f$-compactifiable subsets of $\R^2$. 
\item Define a continuous map $f \colon \R^2 \to \R^2$ by $f(x,y) = (x, y+1)$. 
For each continuous function $\varphi \colon \R \to \R$, 
\[
E_\varphi = \{ (x,y) \in \R^2 \mid y < \varphi(x) \}
\]
is an $f$-compactifiable subset of $\R^2$. 
Note that a pair of two continuous functions $\varphi, \psi \colon \R \rightrightarrows \R$ 
satisfies $E_\varphi \sim_f E_\psi$ if and only if 
$\varphi - \psi$ is a bounded function. 
\end{enumerate}
\end{ex}

\part{The continuous time case}

\section{A continuous time analogue of the Szymczak category}\label{sect:cont-Szymczak}

\begin{defn}
\leavevmode
\begin{enumerate}
\item Let $(X, x_0)$ be a based set and 
\[
F \colon (\R_{\geqslant 0} \times X, \R_{\geqslant 0} \times \{ x_0 \}) \to (X, x_0)
\]
a map. For each $t \in \R_{\geqslant 0}$, 
we write $f^t \colon (X, x_0) \to (X, x_0)$ for the based map defined by $f^t(x) = F(t, x)$. 
We say that $F$ is a \emph{based semiflow} on $X$ if the following two conditions are satisfied: 
\begin{itemize}
\item $f^t f^u = f^{t+u}$ ($t,u \in \R_{\geqslant 0}$). 
\item $f^0 = \id_X$. 
\end{itemize}
\item Let $(X, x_0)$ be a based compact Hausdorff space. 
We say that a based semiflow $F$ on $X$ is \emph{continuous} if 
it is continuous as a map from $\R_{\geqslant 0} \times X$ to $X$. 
\end{enumerate}
\end{defn}

We write $\SFlow(\Set_\ast)$ for the category of based semiflows on sets, i.e.\ 
\begin{itemize}
\item An object of $\SFlow(\Set_\ast)$ is a based semiflow on a set. 
\item A morphism from a based semiflow $F$ on $(X, x_0)$ to a based semiflow $G$ on $(Y, y_0)$ in 
$\SFlow(\Set_\ast)$ is a based map $\varphi \colon (X, x_0) \to (Y, y_0)$ such that 
$\varphi f^t = g^t \varphi$ for any $t \in \R_{\geqslant 0}$. 
\end{itemize}
Similarly, we write $\SFlow(\CHaus_\ast)$ 
for the category of based continuous semiflows on compact Hausdorff spaces, i.e.\ 
\begin{itemize}
\item An object of $\SFlow(\Set_\ast)$ is a based continuous semiflow on a compact Hausdorff space. 
\item A morphism from a based continuous semiflow $F$ on $(X, x_0)$ 
to a based continuous semiflow $G$ on $(Y, y_0)$ in $\SFlow(\CHaus_\ast)$ 
is a based continuous map $\varphi \colon (X, x_0) \to (Y, y_0)$ such that 
$\varphi f^t = g^t \varphi$ for any $t \in \R_{\geqslant 0}$. 
\end{itemize}

To define the point-set level Conley index in the continuous time setting, 
we need an analogue of the Szymczak category, 
which uses the category of based semiflows instead of the category of endomorphisms. 

\begin{defn}
\leavevmode
\begin{enumerate}
\item We define a category $\Sz_\cont(\Set_\ast)$ as follows: 
\begin{itemize}
\item An object of $\Sz_\cont(\Set_\ast)$ is a based semiflow on a set. 
\item For each two based semiflows $F$ and $G$ on sets, 
the hom-set $\Sz_\cont(\Set_\ast)(F, G)$ is given by
\[
\Sz_\cont(\Set_\ast)(F, G) = \left( \SFlow(\Set_\ast)(F, G) \times \R_{\geqslant 0} \right) / \sim,
\]
where $(\varphi, t) \sim (\varphi', t')$ if and only if there exists $s \in \R_{\geqslant 0}$ 
such that $\varphi f^{t'+s} = \varphi' f^{t+s}$ 
(or equivalently, $g^{t'+s} \varphi = g^{t+s} \varphi'$). 
\item The composition of two morphisms 
$\overline{(\varphi, t)} \in \Sz_\cont(\Set_\ast)(F, G)$ and 
$\overline{(\psi, u)} \in \Sz_\cont(\Set_\ast)(G, H)$ 
is given by
$\overline{(\psi, u)} \circ \overline{(\varphi, t)} = \overline{(\psi \varphi, t + u)}$, 
which does not depend on the choice of the representatives. 
\end{itemize}
\item We define a category $\Sz_\cont(\CHaus_\ast)$ as follows: 
\begin{itemize}
\item An object of $\Sz_\cont(\CHaus_\ast)$ 
is a based continuous semiflow on a compact Hausdorff space. 
\item For each two based continuous semiflows $F$ and $G$ on compact Hausdorff spaces, 
the hom-set $\Sz_\cont(\CHaus_\ast)(F, G)$ is given by
\[
\Sz_\cont(\CHaus_\ast)(F, G) = 
\left( \SFlow(\CHaus_\ast)(F, G) \times \R_{\geqslant 0} \right) / \sim,
\]
where $(\varphi, t) \sim (\varphi', t')$ if and only if there exists $s \in \R_{\geqslant 0}$ 
such that $\varphi f^{t'+s} = \varphi' f^{t+s}$ 
(or equivalently, $g^{t'+s} \varphi = g^{t+s} \varphi'$). 
\item The composition of two morphisms 
$\overline{(\varphi, t)} \in \Sz_\cont(\Set_\ast)(F, G)$ and 
$\overline{(\psi, u)} \in \Sz_\cont(\Set_\ast)(G, H)$ 
is given by
$\overline{(\psi, u)} \circ \overline{(\varphi, t)} = \overline{(\psi \varphi, t + u)}$, 
which does not depend on the choice of the representatives. 
\end{itemize}
\end{enumerate}
\end{defn}

\section{Partial semiflows on sets}\label{sect:cont-set}

\begin{defn}
Let $X$ be a set and 
$F \colon \R_{\geqslant 0} \times X \partialto X$ a partial map. 
For each $t \in \R_{\geqslant 0}$, 
we write $f^t \colon X \partialto X$ for the partial map defined by 
\[
\Dom f^t = \{ x \in X \mid (t,x) \in \Dom F \}, \qquad f^t(x) = F(t,x). 
\]
We say that $F$ is a \emph{partial semiflow on $X$} if the following two conditions are satisfied: 
\begin{itemize}
\item $f^t f^u = f^{t+u}$ ($t,u \in \R_{\geqslant 0}$). 
\item $f^0 = \id_X$. 
\end{itemize}
\end{defn}

\begin{defn}
We define a category $\SFlow(\Set^\partial)$ as follows: 
\begin{itemize}
\item An object of $\SFlow(\Set^\partial)$ is a partial semiflow on a set. 
\item A morphism from a partial semiflow $F$ on $X$ to a partial semiflow $G$ on $Y$ in 
$\SFlow(\Set^\partial)$ is a partial map $\varphi \colon X \partialto Y$ such that 
$\varphi f^t = g^t \varphi$ for any $t \in \R_{\geqslant 0}$. 
\end{itemize}
\end{defn}

\begin{defn}
We define a functor 
$({-})^+ \colon \SFlow(\Set^\partial) \to \SFlow(\Set_\ast)$ 
by assigning to each partial semiflow $F$ on $X$ the based semiflow $F^+$ on $X^+$ defined by 
\[
F^+(t,x) = 
\begin{cases}
F(t,x) & \text{(if $(t,x) \in \Dom F$)}, \\ 
\ast & \text{(otherwise)}. 
\end{cases}
\]
\end{defn}

\begin{prop}\label{prop:cont-plus-1}
$({-})^+ \colon \SFlow(\Set^\partial) \to \SFlow(\Set_\ast)$ is an equivalence of categories. 
\end{prop}
\begin{proof}
Define a functor $U \colon \SFlow(\Set_\ast) \to \SFlow(\Set^\partial)$ by 
assigning to each based semiflow $F$ on $(X, x_0)$ the partial semiflow $UF$ on 
$X \smallsetminus \{ x_0 \}$ defined by 
\[
\Dom (UF) = F^{-1}(X \smallsetminus \{ x_0 \}), \qquad (UF)(t,x) = F(t,x). 
\]
Then, $U$ is an inverse of $({-})^+$. 
\end{proof}

\begin{defn}\label{defn:cont-induced}
Let $F \colon \R_{\geqslant 0} \times X \partialto X$ be a partial semiflow on a set $X$. 
For each subset $E$ of $X$, we define a partial semiflow 
$F_E \colon \R_{\geqslant 0} \times E \partialto E$ on $E$ by 
\[
\Dom F_E = \left\{ (t,x) \in \R_{\geqslant 0} \times X \ \middle| \ 
x \in \bigcap_{s \in [0,t]} f^{-s}(E) \right\},
\qquad F_E(t,x) = F(t,x), 
\]
and call it the \emph{induced partial semiflow on $E$}. 
\end{defn}

In this paper, by $f_E^t \ (t \in \R_{\geqslant 0})$, we always mean $(f_E)^t$ 
(i.e.\ the time-$t$ partial map for the partial semiflow $F_E$), not $(f^t)_E$. 
Thus,
\[
\Dom f_E^t = \{ x \in X \mid (t,x) \in \Dom F_E \} 
= \bigcap_{s \in [0,t]} f^{-s}(E), \qquad f_E^t(x) = f^t(x). 
\]

\begin{defn}
Let $F \colon \R_{\geqslant 0} \times X \partialto X$ be a partial semiflow on a set $X$. 
Let $E$ and $E'$ be two subsets of $X$. 
\begin{enumerate}
\item We say that a triple $(a,b,c) \in \R_{\geqslant 0}^3$ is 
\emph{$(E, E')$-admissible relative to $F$} if 
\[
\bigcap_{t \in [0, a+b]} f^{-t}(E) \subset f^{-a}(E'), \qquad 
\bigcap_{t' \in [0, b+c]} f^{-t'}(E') \subset f^{-b}(E). 
\]
We write $A_F(E, E') \subset \R_{\geqslant 0}^3$ 
for the set of all $(E, E')$-admissible triples relative to $F$. 
\item We use the notation $E \sim_F E'$ to mean $A_F(E, E') \neq \varnothing$. 
\end{enumerate}
\end{defn}

\begin{rmk}\label{rmk:cont-sim-symmetric}
The relation $E \sim_F E'$ holds if and only if there exist 
$a, a', b, b' \in \R_{\geqslant 0}$ such that 
\[
\bigcap_{t \in [0, a+b]} f^{-t}(E) \subset f^{-a}(E'), \qquad 
\bigcap_{t' \in [0, a'+b']} f^{-t'}(E') \subset f^{-a'}(E). 
\]
\end{rmk}

\begin{lem}\label{lem:cont-sim-equivalence-relation}
Let $F \colon \R_{\geqslant 0} \times X \partialto X$ be a partial semiflow on a set $X$. 
\begin{enumerate}
\item For any subset $E$ of $X$, we have 
$A_F(E, E) = \R_{\geqslant 0}^3$. 
In particular, $(0,0,0) \in A_F(E, E)$. 
\item Let $E, E'$, and $E''$ be three subsets of $X$. 
If $(a,b,c) \in A_F(E, E')$ and $(a',b',c') \in A_F(E', E'')$, 
then $(a+a',b+b',c+c') \in A_F(E, E'')$. 
\end{enumerate}
\end{lem}

\begin{proof}
The same as the proof of Lemma~\ref{lem:disc-sim-equivalence-relation}. 
\end{proof}

It follows from Remark~\ref{rmk:cont-sim-symmetric} and 
Lemma~\ref{lem:cont-sim-equivalence-relation} that $\sim_F$ is an equivalence relation 
on the set of all subsets of $X$. 

\begin{defn}
Let $F \colon \R_{\geqslant 0} \times X \partialto X$ be a partial semiflow on a set $X$. 
Let $E$ and $E'$ be two subsets of $X$. 
For each $(a,b,c) \in A_F(E, E')$, define a partial map 
$f_{E'E}^{(a,b,c)} \colon E \partialto E'$ by 
\begin{align*}
\Dom f_{E'E}^{(a,b,c)} 
&= \bigcap_{t \in [0, a+b]} f^{-t}(E) \cap \bigcap_{t' \in [a, a+b+c]} f^{-t'}(E'), \\
f_{E'E}^{(a,b,c)}(x) &= f^{a+b+c}(x). 
\end{align*}
\end{defn}

\begin{thm}\label{thm:cont-conley-1}
Let $F \colon \R_{\geqslant 0} \times X \partialto X$ be a partial semiflow on a set $X$. 
\begin{enumerate}
\item Let $E$ be a subset of $X$. 
Then, for any $(a,b,c) \in \R_{\geqslant 0}^3$, we have 
\[
f_{EE}^{(a,b,c)} = f_E^{a+b+c}.
\] 
In particular, $f_{EE}^{(0,0,0)} = \id_E$. 
\item Let $E, E'$, and $E''$ be three subsets of $X$. 
Then, for any $(a,b,c) \in A_F(E, E')$ and $(a',b',c') \in A_F(E', E'')$, 
we have 
\[
f_{E''E'}^{(a',b',c')} \circ f_{E'E}^{(a,b,c)} = f_{E''E}^{(a+a',b+b',c+c')}. 
\]
\item Let $E$ and $E'$ be two subsets of $X$. 
Then, for any $(a,b,c) \in A_F(E, E')$ and any $t \in \R_{\geqslant 0}$, we have 
\[
f_{E'E}^{(a,b,c)} \circ f_E^t = f_{E'}^t \circ f_{E'E}^{(a,b,c)}. 
\]
\item Let $E$ and $E'$ be two subsets of $X$. 
Then, for any $(a,b,c), (a',b',c') \in A_F(E, E')$, we have 
\[
f_{E'E}^{(a,b,c)} \circ f_E^{a'+b'+c'} = f_{E'E}^{(a',b',c')} \circ f_E^{a+b+c}. 
\]
\end{enumerate}
\end{thm}
\begin{proof}
The same as the proof of Theorem~\ref{thm:disc-conley-1}. 
\end{proof}

\begin{defn}
Let $F \colon \R_{\geqslant 0} \times X \partialto X$ be a partial semiflow on a set $X$. 
\begin{enumerate}
\item We define a small category $\widetilde{\Sub}_F$ as follows: 
\begin{itemize}
\item An object of $\widetilde{\Sub}_F$ is a subset of $X$. 
\item For each two subsets $E$ and $E'$ of $X$, 
the hom-set $\widetilde{\Sub}_F(E, E')$ is given by $\widetilde{\Sub}_F(E, E') = A_F(E, E')$. 
\item The composition of two morphisms is given by addition. 
\end{itemize}
\item We define a functor $\widetilde{\Conl}_F \colon \widetilde{\Sub}_F \to \SFlow(\Set_\ast)$ 
as follows: 
\begin{itemize}
\item For each subset $E$ of $X$, we put 
$\widetilde{\Conl}_F(E) = F_E^+$. 
\item For each two subsets $E$ and $E'$ of $X$, we assign to $(a,b,c) \in A_F(E, E')$ 
the morphism $(f_{E'E}^{(a,b,c)})^+ \colon F_E^+ \to F_{E'}^+$.
\end{itemize}
\end{enumerate}
\end{defn}

As in Definition~\ref{defn:disc-functor-1}, 
we can see that $\widetilde{\Sub}_F$ is a category and $\widetilde{\Conl}_F$ is a functor. 

\begin{defn}\label{defn:cont-functor-2}
Let $F \colon \R_{\geqslant 0} \times X \partialto X$ be a partial semiflow on a set $X$. 
\begin{enumerate}
\item We define a small category $\Sub_F$ as follows: 
\begin{itemize}
\item An object of $\Sub_F$ is a subset of $X$. 
\item For each two subsets $E$ and $E'$ of $X$, 
the hom-set $\Sub_F(E, E')$ is given by 
\[
\Sub_F(E, E') = 
\begin{cases}
\ast & \text{(if $E \sim_F E'$)}, \\
\varnothing & \text{(otherwise)}. 
\end{cases}
\]
\end{itemize}
\item We define a functor $\Conl_F \colon \Sub_F \to \Sz_\cont(\Set_\ast)$ as follows: 
\begin{itemize}
\item For each subset $E$ of $X$, we put 
$\Conl_F(E) = F_E^+$. 
\item To each two subsets $E$ and $E'$ of $X$ with $E \sim_F E'$, 
we assign the morphism $f_{E'E}^+ \colon F_E^+ \to F_{E'}^+$ defined by 
\[
f_{E'E}^+ = \overline{\left((f_{E'E}^{(a,b,c)})^+, a+b+c \right)}, 
\]
where $(a,b,c) \in A_F(E, E')$. 
\end{itemize}
\end{enumerate}
\end{defn}

As in Definition~\ref{defn:disc-functor-2}, we can see that $\Conl_F$ is a well-defined functor. 

\begin{rmk}
Since $\sim_F$ is a symmetric relation, every morphism in $\Sub_F$ is an isomorphism. 
\end{rmk}

\section{Continuous partial semiflows on topological spaces}\label{sect:cont-topological-space}

\begin{defn}
Let $X$ be a topological space. We say that a partial semiflow $F$ on $X$ is \emph{continuous} 
if $F$ is continuous as a partial map from $\R_{\geqslant 0} \times X$ to $X$. 
\end{defn}

\begin{defn}
Let $F \colon \R_{\geqslant 0} \times X \partialto X$ 
be a continuous partial semiflow on a topological space $X$. 
\begin{enumerate}
\item We say that $F$ is \emph{finite-time proper} 
if, for any $t \in \R_{\geqslant 0}$, the restriction 
$F|_{[0, t] \times X} \colon [0, t] \times X \partialto X$
is proper. 
\item We say that $F$ is \emph{openly defined} 
if it is openly defined as a continuous partial map from $\R_{\geqslant 0} \times X$ to $X$. 
\end{enumerate}
\end{defn}

\begin{defn}
We define a category $\SFlow(\LCHaus^\partial)$ as follows: 
\begin{itemize}
\item An object of $\SFlow(\LCHaus^\partial)$ is a finite-time proper and openly defined 
partial semiflow on a locally compact Hausdorff space. 
\item A morphism from $F \colon \R_{\geqslant 0} \times X \partialto X$ to
$G \colon \R_{\geqslant 0} \times Y \partialto Y$ in $\SFlow(\LCHaus^\partial)$ 
is a proper and openly defined partial map $\varphi \colon X \partialto Y$ such that 
$\varphi f^t = g^t \varphi$ for any $t \in \R_{\geqslant 0}$. 
\end{itemize}
\end{defn}

The significance of finite-time proper and openly defined partial semiflows is 
that we can define the following functor: 

\begin{defn}
We define a functor 
$({-})^+ \colon \SFlow(\LCHaus^\partial) \to \SFlow(\CHaus_\ast)$ 
by assigning to each finite-time proper and openly defined partial semiflow $F$ on $X$ 
the continuous based semiflow $F^+$ on $X^+$ defined by
\[
F^+(t,x) = 
\begin{cases}
F(t,x) & \text{(if $(t,x) \in \Dom F$)}, \\ 
\ast & \text{(otherwise)}. 
\end{cases}
\]
\end{defn}

Let us verify that $F^+$ is indeed continuous. 
Suppose that $K$ is a closed subset of $X^+$ that does not contain $\ast$. 
For each $t \in \R_{\geqslant 0}$, we have
\[
(F^+)^{-1}(K) \cap ([0, t] \times X) = (F|_{[0, t] \times X})^{-1}(K). 
\]
Since $K$ is a compact subset of $X$ and 
$F|_{[0, t] \times X} \colon \Dom (F|_{[0, t] \times X}) \to X$ is a proper map, 
$(F^+)^{-1}(K) \cap ([0, t] \times X)$ is compact, hence closed in $[0, t] \times X^+$. 
Since $t$ is arbitrary, we conclude that $(F^+)^{-1}(K)$ is closed in $\R_{\geqslant 0} \times X^+$. 
Suppose that $K$ is a closed subset of $X^+$ that contains $\ast$. We have 
\[
(F^+)^{-1}(K) = F^{-1}(K \smallsetminus \{ \ast \}) \cup 
((\R_{\geqslant 0} \times X) \smallsetminus \Dom F) \cup (\R_{\geqslant 0} \times \{ \ast \}). 
\]
Since $K \smallsetminus \{ \ast \}$ is a closed subset of $X$ and 
$\Dom F$ is an open subset of $\R_{\geqslant 0} \times X$, we see that 
$F^{-1}(K \smallsetminus \{ \ast \}) \cup ((\R_{\geqslant 0} \times X) \smallsetminus \Dom F)$ 
is closed in $\R_{\geqslant 0} \times X$. 
Thus, $(F^+)^{-1}(K)$ is closed in $\R_{\geqslant 0} \times X^+$. 

\begin{prop}\label{prop:cont-plus-2}
$({-})^+ \colon \SFlow(\LCHaus^\partial) \to \SFlow(\CHaus_\ast)$ is an equivalence of categories. 
\end{prop}

\begin{proof}
Define a functor $U \colon \SFlow(\CHaus_\ast) \to \SFlow(\LCHaus^\partial)$ by 
assigning to each based continuous semiflow $F$ on $(X, x_0)$ 
the partial semiflow $UF$ on $X \smallsetminus \{ x_0 \}$ defined by 
\[
\Dom (UF) = F^{-1}(X \smallsetminus \{ x_0 \}), \qquad (UF)(t,x) = F(t,x), 
\]
which is finite-time proper and openly defined. 
Then, $U$ is an inverse of $({-})^+$. 

For the sake of completeness, 
let us verify that $UF$ is indeed finite-time proper and openly defined. 
Take any $t \in \R_{\geqslant 0}$. 
We have 
\[
\Dom \left( (UF)|_{[0, t] \times (X \smallsetminus \{ x_0 \})} \right) 
= F^{-1}(X \smallsetminus \{ x_0 \}) \cap ([0, t] \times X). 
\]
Consider the following pullback diagram: 
\[
\begin{tikzcd}
\Dom ((UF)|_{[0, t] \times X}) \arrow[rrr, "(UF)|_{[0, t] \times (X \smallsetminus \{ x_0 \})}"] \arrow[d, hook] & & & X \smallsetminus \{ x_0 \} \arrow[d, hook] \\
{[0,t]} \times X \arrow[rrr, "F"] & & & X.
\end{tikzcd}
\]
The bottom map is proper since 
it is a continuous map from a compact Hausdorff space (Lemma~\ref{lem:proper-source-compact}). 
Thus, the top map is also proper (Lemma~\ref{lem:proper-base-change}). 
Since $X \smallsetminus \{ x_0 \}$ is open in $X$, 
the inverse image $F^{-1}(X \smallsetminus \{ x_0 \}) \ (= \Dom (UF))$ 
is open in $\R_{\geqslant 0} \times X$. 
Since $\Dom (UF)$ is contained in $\R_{\geqslant 0} \times (X \smallsetminus \{ x_0 \})$, 
it is also open in $\R_{\geqslant 0} \times (X \smallsetminus \{ x_0 \})$.
\end{proof}

Let us prove some basic properties of 
finite-time proper partial semiflows and openly defined partial semiflows. 

\begin{lem}\label{lem:F-f}
Let $F \colon \R_{\geqslant 0} \times X \partialto X$ 
be a continuous partial semiflow on a topological space $X$. 
\begin{enumerate}
\item If $F$ is finite-time proper, $f^t$ is proper for any $t \in \R_{\geqslant 0}$. 
\item If $F$ is openly defined, $f^t$ is openly defined for any $t \in \R_{\geqslant 0}$. 
\end{enumerate}
\end{lem}
\begin{proof}
(1) Let $i \colon X \to [0,t] \times X$ 
be the closed inclusion defined by $i(x) = (t, x)$. 
We can factorize $f^t \colon X \partialto X$ as follows: 
\[
X \xrightarrow{i} [0, t] \times X
\xpartialto{F|_{[0, t] \times X}} X. 
\]
Since $i$ and $F|_{[0, t] \times X}$ are both proper, 
their composite $f^t$ is also proper. 

(2) Let $i \colon X \to \R_{\geqslant 0}\times X$ 
be the closed inclusion defined by $i(x) = (t, x)$. 
Since $\Dom f^t = i^{-1}(\Dom F)$, 
we see that $\Dom f^t$ is open in $X$. 
\end{proof}

\begin{lem}\label{lem:short-time}
Let $F \colon \R_{\geqslant 0} \times X \partialto X$ 
be a continuous partial semiflow on a topological space $X$. 
\begin{enumerate}
\item $F$ is finite-time proper if and only if 
$F|_{[0, \varepsilon] \times X}$ is proper for some $\varepsilon \in \R_{>0}$. 
\item $F$ is openly defined if and only if 
$F|_{[0, \varepsilon] \times X}$ is openly defined for some $\varepsilon \in \R_{>0}$. 
\end{enumerate}
\end{lem}

\begin{proof}
(1) The `only if' part is trivial. Let us prove the `if' part. 
For each $n \in \N$, put 
\[
F_n = F|_{[n\varepsilon, (n+1)\varepsilon] \times X} \colon 
[n\varepsilon, (n+1)\varepsilon] \times X \partialto X. 
\]
We first prove that $F_n$ is a proper partial map for any $n \in \N$. 
By assumption, $F_0$ is proper. Suppose that $F_n$ is proper. 
Let $i \colon X \to [0, \varepsilon] \times X$ 
be the closed inclusion defined by $i(x) = (\varepsilon, x)$, 
and let 
\[
g_{n+1} \colon [(n+1)\varepsilon, (n+2)\varepsilon] \to [n\varepsilon, (n+1)\varepsilon]
\]
be the homeomorphism defined by $g_{n+1}(t) = t-\varepsilon$. 
Then, we can factorize $F_{n+1}$ as follows: 
\[
[(n+1)\varepsilon, (n+2)\varepsilon] \times X
\xrightarrow{g_{n+1} \times \id_X} [n\varepsilon,(n+1)\varepsilon] \times X 
\xpartialto{F_n} X
\xrightarrow{i} [0, \varepsilon] \times X 
\xpartialto{F_0} X. 
\]
This shows that $F_{n+1}$ is proper. The induction is completed. 

Take any $s \in \R_{\geqslant 0}$. Choose $N \in \N$ so that $N\varepsilon \geqslant s$. 
Patching $F_n$ for $0 \leqslant n \leqslant N-1$, we see that 
\[
F|_{[0, N\varepsilon] \times X} \colon [0, N\varepsilon] \times X \partialto X
\]
is proper (Lemma~\ref{lem:proper-finite-closed-cover}). 
Now, $F|_{[0,s] \times X} \colon [0,s] \times X \partialto X$ is factorized as 
\[
[0,s] \times X \hookrightarrow 
[0,N\varepsilon] \times X \xpartialto{F|_{[0,N\varepsilon] \times X}} X, 
\]
and thus proper. 

(2) The `only if' part is trivial. Let us prove the `if' part. 
As in the proof of (1), we can prove that 
\[
F_n = F|_{[n\varepsilon, (n+1)\varepsilon] \times X} \colon 
[n\varepsilon, (n+1)\varepsilon] \times X \partialto X. 
\]
is openly defined for any $n \in \N$. 
Patching $F_n$ for all $n \in \N$, we see that $F$ is openly defined. 
\end{proof}

\begin{thm}\label{thm:cont-conley-2}
Let $F \colon \R_{\geqslant 0} \times X \partialto X$ 
be a continuous partial semiflow on a topological space $X$. 
Let $E$ and $E'$ be two subsets of $X$. Let $(a,b,c) \in A_F(E, E')$. 
\begin{enumerate}
\item If $F_E$ and $F_{E'}$ are finite-time proper, 
$f_{E'E}^{(a,b,c)}$ is proper. 
\item If $F_E$ and $F_{E'}$ are openly defined, $f_{E'E}^{(a,b,c)}$ is openly defined. 
\end{enumerate}
\end{thm}
\begin{proof}
The same as the proof of Theorem~\ref{thm:disc-conley-2}. 
We use Lemma~\ref{lem:F-f}. 
\end{proof}

\begin{defn}
Let $F \colon \R_{\geqslant 0} \times X \partialto X$ 
be a continuous partial semiflow on a locally compact Hausdorff space $X$. 
We say that a locally closed subset $E$ of $X$ is \emph{$F$-compactifiable} 
if the induced partial semiflow $F_E$ on $E$ is finite-time proper and openly defined.
\end{defn}

Let $\widetilde{\CptifSub}_F$ (resp.\ $\CptifSub_F$) be the full category of $\widetilde{\Sub}_F$ 
(resp.\ $\Sub_F$) consisting of $F$-compactifiable subsets of $X$. 
By Proposition~\ref{prop:cont-plus-2} and Theorem~\ref{thm:cont-conley-2}, 
the functor $\widetilde{\Conl}_F \colon \widetilde{\Sub}_F \to \SFlow(\Set_\ast)$ 
(resp.\ $\Conl_F \colon \Sub_F \to \Sz_\cont(\Set_\ast)$) 
induces a functor from $\widetilde{\CptifSub}_F$ to $\SFlow(\CHaus_\ast)$ 
(resp.\ from $\CptifSub_F$ to $\Sz_\cont(\CHaus_\ast)$). 

\begin{defn}
Let $F \colon \R_{\geqslant 0} \times X \partialto X$ 
be a continuous partial semiflow on a locally compact Hausdorff space $X$. 
We use the same notations 
$\widetilde{\Conl}_F \colon \widetilde{\CptifSub}_F \to \SFlow(\CHaus_\ast)$ and 
$\Conl_F \colon \CptifSub_F \to \Sz_\cont(\CHaus_\ast)$ for the functors defined above. 
\end{defn}

\section{Some lemmas}\label{sect:cont-lemmas}

In this section, we prepare some lemmas needed in the next section. 

\begin{lem}\label{lem:induced-compact-open}
Let $F \colon \R_{\geqslant 0} \times X \partialto X$ 
be a continuous partial semiflow on a topological space $X$. 
\begin{enumerate}
\item Let $K$ be a compact Hausdorff subset of $X$ such that 
$[0,\varepsilon] \times K \subset \Dom F$ for some $\varepsilon \in \R_{>0}$. 
Then, the induced partial semiflow
$F_K \colon \R_{\geqslant 0} \times K \partialto K$ is finite-time proper.
\item Let $U$ be an open subset of $X$ such that
$[0,\varepsilon] \times U \subset \Dom F$ for some $\varepsilon \in \R_{>0}$. 
Then, the induced partial semiflow 
$F_U \colon \R_{\geqslant 0} \times U \partialto U$ is openly defined. 
\end{enumerate}
\end{lem}

\begin{proof}
(1) By Lemma~\ref{lem:short-time}~(1), it suffices to verify that 
\[
F_K|_{[0,\varepsilon] \times K} \colon [0,\varepsilon] \times K \partialto K
\]
is proper. By Lemma~\ref{lem:proper-source-compact}, to see that this partial map is proper, 
it is enough to verify that $\Dom (F_K|_{[0,\varepsilon] \times K})$ is compact Hausdorff. 
Since $[0,\varepsilon] \times K$ is compact Hausdorff, it suffices to see that 
$\Dom (F_K|_{[0,\varepsilon] \times K})$ is closed in $[0,\varepsilon] \times K$.
Define a closed subset $\Delta_\varepsilon$ of $[0,\varepsilon] \times [0,\varepsilon]$ and 
two continuous maps $\pi_1, \pi_2 \colon \Delta_\varepsilon \rightrightarrows [0,\varepsilon]$ by 
\[
\Delta_\varepsilon 
= \{ (s,t) \in [0,\varepsilon] \times [0,\varepsilon] \mid s \leqslant t \}, \qquad 
\pi_1(s,t) = s, \quad \pi_2(s,t) = t. 
\]
Consider the following diagram: 
\[
\begin{tikzcd}
\Delta_\varepsilon \times K \arrow[rr, "\pi_1 \times \id_K"] \arrow[d, "\pi_2 \times \id_K" left] 
& & {[0,\varepsilon] \times K} \arrow[rr, "F|_{[0,\varepsilon] \times K}"] & & X. \\
{[0,\varepsilon] \times K} & & & & 
\end{tikzcd}
\]
Then, we have 
\begin{align*}
&\Dom (F_K|_{[0,\varepsilon] \times K}) \\
&\qquad= ([0,\varepsilon] \times K) \smallsetminus (\pi_2 \times \id_K)((\pi_1 \times \id_K)^{-1}(([0,\varepsilon] \times K) \smallsetminus F^{-1}(K))).
\end{align*}
Since $\pi_2 \colon \Delta_\varepsilon \to [0,\varepsilon]$ is an open map, 
the base change $\pi_2 \times \id_X$ is also an open map (Lewis~\cite[Lem.~1.6]{Lew85}). 
Thus, $\Dom (F_K|_{[0,\varepsilon] \times K})$ is closed in $[0,\varepsilon] \times K$.

(2) By Lemma~\ref{lem:short-time}~(2), it suffices to verify that 
\[
\Dom (F_U|_{[0,\varepsilon] \times U}) = 
\left\{ (t,x) \in [0,\varepsilon] \times U \ \middle| \ x \in \bigcap_{s \in [0,t]} f^{-s}(E) \right\}
\]
is open in $[0,\varepsilon] \times U$. Let $\Delta_\varepsilon$ and 
$\pi_1, \pi_2 \colon \Delta_\varepsilon \rightrightarrows [0,\varepsilon]$ 
be as in the proof of (1). Consider the following diagram: 
\[
\begin{tikzcd}
\Delta_\varepsilon \times U \arrow[rr, "\pi_1 \times \id_U"] \arrow[d, "\pi_2 \times \id_U" left] 
& & {[0,\varepsilon] \times U} \arrow[rr, "F|_{[0,\varepsilon] \times U}"] & & X. \\
{[0,\varepsilon] \times U} & & & & 
\end{tikzcd}
\]
Then, we have 
\begin{align*}
&\Dom (F_U|_{[0,\varepsilon] \times U}) \\
&\qquad= ([0,\varepsilon] \times U) \smallsetminus (\pi_2 \times \id_U)((\pi_1 \times \id_U)^{-1}(([0,\varepsilon] \times U) \smallsetminus F^{-1}(U))).
\end{align*}
Since $\pi_2 \colon \Delta_\varepsilon \to [0,\varepsilon]$ is a proper map, 
the base change $\pi_2 \times \id_X$ is a closed map (Lemma~\ref{lem:proper-universally-closed}). 
Thus, $\Dom (F_U|_{[0,\varepsilon] \times U})$ is open in $[0,\varepsilon] \times U$.
\end{proof}

\begin{lem}\label{lem:intersection-compact}
Let $F \colon \R_{\geqslant 0} \times X \partialto X$ 
be a continuous partial semiflow on a topological space $X$. 
Let $K$ be a compact Hausdorff subset of $X$ such that 
$[0,\varepsilon] \times K \subset \Dom F$ for some $\varepsilon \in \R_{>0}$. 
Then, for any $a,b \in \R_{\geqslant 0}$, 
\[
f^a\left( \bigcap_{t \in [0,a+b]} f^{-t}(K) \right)
\]
is a compact Hausdorff subset of $X$. 
\end{lem}
\begin{proof}
By Lemma~\ref{lem:induced-compact-open}~(1), the induced partial semiflow 
$F_K \colon \R_{\geqslant 0} \times K \partialto K$ is finite-time proper. 
By Lemma~\ref{lem:F-f}~(1), the induced partial self-map $f_K^{a+b} \colon K \partialto K$ 
is proper. Since $K$ is compact Hausdorff, the inverse image 
\[
f_K^{-(a+b)}(K) = \Dom f_K^{a+b} = \bigcap_{t \in [0,a+b]} f^{-t}(K)
\]
must be compact Hausdorff. Thus, 
\[
f^a\left( \bigcap_{t \in [0,a+b]} f^{-t}(K) \right)
\]
is also compact Hausdorff. 
\end{proof}

\begin{rmk}\label{rmk:short-time-mild}
The assumption that $[0,\varepsilon] \times K \subset \Dom F$ 
(resp.\ $[0,\varepsilon] \times U \subset \Dom F$) in 
Lemmas~\ref{lem:induced-compact-open}~(1) and \ref{lem:intersection-compact} 
(resp.\ Lemma~\ref{lem:induced-compact-open}~(2)) is very mild. 
For example, this trivially holds when $F$ is an everywhere defined map. 
We also have the following result:
\end{rmk}

\begin{lem}\label{lem:short-time-compact}
Let $F \colon \R_{\geqslant 0} \times X \partialto X$ 
be a continuous partial semiflow on a topological space $X$. 
Suppose that $F$ is openly defined. 
Then, for any relatively compact subset $E$ of $X$, 
there exists $\varepsilon \in \R_{>0}$ such that $[0,\varepsilon] \times E \subset \Dom F$. 
\end{lem}
\begin{proof}
Let $K = \overline{E}$ be the closure of $E$. It suffices to prove that 
$[0, \varepsilon] \times K \subset \Dom F$ for some $\varepsilon \in \R_{>0}$. 
Since $\Dom F$ is open in $\R_{\geqslant 0} \times X$ and contains $\{ 0 \} \times X$, 
we have 
\[
\bigcup_{t \in \R_{>0}} \Dom f^t = X \ (\supset K). 
\]
It follows from Lemma~\ref{lem:F-f}~(2) that $\Dom f^t$ is open in $X$ for each $t \in \R_{>0}$. 
Since $K$ is compact and the family $(\Dom f^t)_{t \in \R_{>0}}$ 
is decreasing with respect to $t$, 
there exists $\varepsilon \in \R_{>0}$ such that $K \subset \Dom f^\varepsilon$. 
This is equivalent to saying that $[0,\varepsilon] \times K \subset \Dom F$. 
\end{proof}

\section{The Conley index of isolated $F$-invariant subsets}\label{sect:cont-isolated}

\begin{defn}
Let $F \colon \R_{\geqslant 0} \times X \partialto X$ be a partial semiflow on a set $X$. 
We say that a subset $S$ of $X$ is \emph{$F$-invariant} 
if $S \subset \Dom f^t$ and $f^t(S) = S$ hold for any $t \in \R_{\geqslant 0}$. 
\end{defn}

\begin{defn}
Let $F \colon \R_{\geqslant 0} \times X \partialto X$ be a partial semiflow on a set $X$. 
Let $E$ be a subset of $X$. 
We define the \emph{$F$-invariant part} $I_F(E)$ of $E$ by 
\[
I_F(E) = \bigcap_{a,b \in \R_{\geqslant 0}} f^a \left( \bigcap_{t \in [0,a+b]} f^{-t}(E) \right).
\]
\end{defn}

\begin{rmk}\label{rmk:cont-decreasing}
The family 
\[
\left( f^a\left( \bigcap_{t \in [0,a+b]} f^{-t}(E) \right) \right)_{a,b \in \R_{\geqslant 0}}
\]
is decreasing in the following sense: 
for any $a,b,a',b' \in \R_{\geqslant 0}$ satisfying $a \leqslant a'$ and $b \leqslant b'$, we have
\[
f^{a'}\left( \bigcap_{t \in [0,a'+b']} f^{-t}(E) \right) 
\subset f^a\left( \bigcap_{t \in [0,a+b]} f^{-t}(E) \right).
\]
\end{rmk}

The above definition of the $F$-invariant part agrees with the usual one, namely: 

\begin{lem}\footnote{This Lemma is wrong. See Correction.}
Let $F \colon \R_{\geqslant 0} \times X \partialto X$ be a partial semiflow on a set $X$. 
Let $E$ be a subset of $X$. 
Then, $I_F(E)$ is the largest $F$-invariant subset contained in $E$. 
\end{lem}

\begin{proof}
The same as the proof of Lemma~\ref{lem:disc-invariant-part}. 
\end{proof}

\begin{defn}\label{defn:cont-isolated}
Let $F \colon \R_{\geqslant 0} \times X \partialto X$ 
be a continuous partial semiflow on a locally compact Hausdorff space $X$. 
Let $S$ be an $F$-invariant subset of $X$. 
\begin{enumerate}
\item A neighbourhood $E$ of $S$ is called \emph{isolating} if 
the following three conditions are satisfied: 
\begin{itemize}
\item $E$ is relatively compact. 
\item There exists $\varepsilon \in \R_{>0}$ such that 
$[0, \varepsilon] \times \overline{E} \subset \Dom F$. 
\item $S = I_F(\overline{E})$.
\end{itemize} 
Here, $\overline{E}$ denotes the closure of $E$. 
\item $S$ is called \emph{isolated} if it admits an isolating neighbourhood. 
\end{enumerate}
\end{defn}

\begin{rmk}
By Lemma~\ref{lem:short-time-compact}, if $F$ is openly defined, 
the second condition in Definition~\ref{defn:cont-isolated}~(1) is automatically satisfied. 
\end{rmk}

\begin{rmk}\label{rmk:cont-isolating-closure-smaller}
If $E$ is an isolating neighbourhood of $S$, 
the closure $\overline{E}$ is also isolating. 
If $E$ is an isolating neighbourhood of $S$ 
and $E'$ is a neighbourhood of $S$ contained in $E$,
then $E'$ is also isolating. 
\end{rmk}

\begin{lem}\label{lem:cont-compact}
Let $F \colon \R_{\geqslant 0} \times X \partialto X$ 
be a continuous partial semiflow on a locally compact Hausdorff space $X$. 
Let $S$ be an isolated $F$-invariant subset of $X$. 
Then, $S$ is compact. 
\end{lem}
\begin{proof}
The same as the proof of Lemma~\ref{lem:disc-compact}. 
We use Lemma~\ref{lem:intersection-compact}. 
\end{proof}

\begin{thm}\label{thm:cont-isolating-1}
Let $F \colon \R_{\geqslant 0} \times X \partialto X$ 
be a continuous partial semiflow on a locally compact Hausdorff space $X$. 
Let $S$ be an isolated $F$-invariant subset of $X$. 
Then, for any two isolating neighbourhoods $E$ and $E'$ of $S$, we have $E \sim_F E'$. 
\end{thm}
\begin{proof}
The same as the proof of Theorem~\ref{thm:disc-isolating-1}. 
We use Lemma~\ref{lem:intersection-compact}. 
\end{proof}

\begin{defn}
Let $F \colon \R_{\geqslant 0} \times X \partialto X$ 
be a continuous partial semiflow on a locally compact Hausdorff space $X$. 
Let $S$ be an $F$-invariant subset of $X$. 
We say that a neighbourhood $E$ of $S$ is an \emph{index neighbourhood} if 
it is isolating and $F$-compactifiable. 
\end{defn}

\begin{thm}\label{thm:cont-isolating-2}
Let $F \colon \R_{\geqslant 0} \times X \partialto X$ 
be a continuous partial semiflow on a locally compact Hausdorff space $X$. 
Let $S$ be an isolated $F$-invariant subset of $X$. 
Then, the set of index neighbourhoods of $S$ forms a neighbourhood base for $S$
(in particular, there exists an index neighbourhood of $S$). 
\end{thm}

\begin{proof}
The proof is similar to that of Theorem~\ref{thm:disc-isolating-2}, 
although slightly more involved.

Take any neighbourhood $N$ of $S$. 
Since $S$ is compact by Lemma~\ref{lem:cont-compact} and $X$ is locally compact Hausdorff, 
we can take a compact neighbourhood $K$ of $S$ contained in $N$. 
Let $U = K^\circ$ the interior of $K$. 
By Remark~\ref{rmk:cont-isolating-closure-smaller}, 
we may assume that $K$ and $U$ are isolating neighbourhoods of $S$. 
By Theorem~\ref{thm:cont-isolating-1}, we have $K \sim_F U$. 
Take $(a,b,c) \in A_F(K, U)$ arbitrarily, and put 
\[
E = \bigcap_{t \in [0,a+b]} f^{-t}(K) \cap \bigcap_{t' \in [a,a+b+c]} f^{-t'}(U). 
\]
Then, $E$ is an isolating neighbourhood of $S$. Indeed, 
\[
U' = \bigcap_{t \in [0,a+b+c]} f^{-t}(U)
\]
is open in $X$ by Lemmas~\ref{lem:F-f}~(2) and \ref{lem:induced-compact-open}~(2), 
and $S \subset U' \subset E$ holds. 
We have $E \subset N$. 
By Lemma~\ref{lem:induced-compact-open}~(1),
the induced partial semiflow $F_K \colon \R_{\geqslant 0} \times K \partialto K$ 
is finite-time proper. 
By Lemma~\ref{lem:induced-compact-open}~(2), 
the induced partial semiflow $F_U \colon \R_{\geqslant 0} \times U \partialto U$ is openly defined. 
Also, $E$ is locally closed in $X$. 
Now, it is enough to prove the following lemma: 

\begin{lem}\label{lem:cont-compactifiable-construction}
Let $F \colon \R_{\geqslant 0} \times X \partialto X$ 
be a continuous partial semiflow on a topological space $X$. 
Let $E$ and $E'$ be two subsets of $X$ and $(a,b,c) \in A_F(E, E')$. 
Assume that $F_E \colon \R_{\geqslant 0} \times E \partialto E$ is finite-time proper and 
$F_{E'} \colon \R_{\geqslant 0} \times E' \partialto E'$ is openly defined. Put 
\[
E'' = \bigcap_{t \in [0,a+b]} f^{-t}(E) \cap \bigcap_{t' \in [a,a+b+c]} f^{-t'}(E').
\]
Then, the induced partial semiflow 
$F_{E''} \colon \R_{\geqslant 0} \times E'' \partialto E''$ 
is finite-time proper and openly defined. 
\end{lem}

\begin{rmk}
In the setting of Lemma~\ref{lem:cont-compactifiable-construction}, 
$E''$ satisfies $E'' \sim_F E$ (and hence, $E'' \sim_F E'$). 
Indeed, we have $E'' \subset E$ and 
\[
\bigcap_{t \in [0,a+2b+c]} f^{-t}(E) \subset E''. 
\]
\end{rmk}

\begin{proof}[Proof of Lemma~\ref{lem:cont-compactifiable-construction}]
(1) Let us prove that $F_{E''}$ is finite-time proper. 
Take any $s \in \R_{\geqslant 0}$. Consider the following commutative diagram:
\[
\begin{tikzcd}
\Dom (F_{E''}|_{[0,s] \times E''}) \arrow[rr, "F_{E''}|_{[0,s] \times E''}"] \arrow[d, hook] & & 
E'' \arrow[d, hook] \\
\Dom (F_E|_{[0,s] \times E}) \arrow[rr, "F_E|_{[0,s] \times E}"] & & E. 
\end{tikzcd}
\]
Since the bottom map is proper, in order to see that the top map is proper, 
it is enough to verify that the square is a pullback, i.e.\ 
$(F_E|_{[0,s] \times E})^{-1}(E'') = \Dom (F_{E''}|_{[0,s] \times E''})$ 
(Lemma~\ref{lem:proper-base-change}). We see that 
\begin{align*}
&(F_E|_{[0,s] \times E})^{-1}(E'')
= \left\{ (t,x) \in [0,s] \times X \ \middle| \ x \in \bigcap_{u \in [0,t]} f^{-u}(E) \cap f^{-t}(E'') \right\} \\
&= \left\{ (t,x) \in [0,s] \times X \ \middle| \ x \in \bigcap_{u \in [0,t+a+b]} f^{-u}(E) \cap \bigcap_{u' \in [t+a,t+a+b+c]} f^{-u'}(E') \right\}
\end{align*}
and
\begin{align*}
&\Dom (F_{E''}|_{[0,s] \times E''})
= \left\{ (t, x) \in [0,s] \times X \ \middle| \ x \in \bigcap_{u'' \in [0,t]} f^{-u''}(E'') \right\} \\
&= \left\{ (t, x) \in [0,s] \times X \ \middle| \ x \in \bigcap_{u \in [0,t+a+b]} f^{-u}(E) \cap \bigcap_{u' \in [a,t+a+b+c]} f^{-u'}(E') \right\}.
\end{align*}
Since 
\begin{align*}
\bigcap_{u \in [0,t+a+b]} f^{-u}(E) 
&= \bigcap_{v' \in [0,t]} f^{-v'}\left( \bigcap_{v \in [0,a+b]} f^{-v} (E) \right) \\
&\subset \bigcap_{v' \in [0,t]} f^{-v'}(f^{-a}(E'))
= \bigcap_{u' \in [a,t+a]} f^{-u'}(E'), 
\end{align*}
we obtain 
\begin{align*}
\bigcap_{u \in [0,t+a+b]} f^{-u}(E) 
= \bigcap_{u \in [0,t+a+b]} f^{-u}(E) \cap \bigcap_{u' \in [a,t+a]} f^{-u'}(E')
\end{align*}
and
\begin{align*}
&\bigcap_{u \in [0,t+a+b]} f^{-u}(E) \cap \bigcap_{u' \in [t+a,t+a+b+c]} f^{-u'}(E') \\
&\qquad\qquad\qquad\qquad\qquad\qquad= \bigcap_{u \in [0,t+a+b]} f^{-u}(E) \cap \bigcap_{u' \in [a,t+a+b+c]} f^{-u'}(E').
\end{align*}
We have thus proved 
$(F_E|_{[0,s] \times E})^{-1}(E'') = \Dom (F_{E''}|_{[0,s] \times E''})$. 

Let us prove that $F_{E''}$ is openly defined. 
Since $E'' \subset f^{-(a+b+c)}(E')$, we can define a continuous map 
$g \colon \R_{\geqslant 0} \times E'' \to \R_{\geqslant 0} \times E'$ 
by $g(t,x) = (t, f^{a+b+c}(x))$. Consider the following commutative diagram: 
\[
\begin{tikzcd}
\Dom F_{E''} \arrow[r, hook] \arrow[d, "g" left] & \R_{\geqslant 0} \times E'' \arrow[d, "g"] \\
\Dom F_{E'} \arrow[r, hook] & \R_{\geqslant 0} \times E'.
\end{tikzcd}
\]
Since $\Dom F_{E'}$ is open in $E'$, 
in order to see that $\Dom F_{E''}$ is open in $E''$, 
it is enough to verify that the square is a pullback, 
i.e.\ $g^{-1}(\Dom F_{E'}) = \Dom F_{E''}$. 

We see that 
\begin{align*}
&g^{-1}(\Dom F_{E'}) = g^{-1} \left( \left\{ (t,x) \in \R_{\geqslant 0}\times X \ \middle| \ x \in \bigcap_{u' \in [0,t]} f^{-u'}(E') \right\} \right) \\
&\qquad= \left\{ (t,x) \in \R_{\geqslant 0}\times E'' \ \middle| \ x \in f^{-(a+b+c)} \left( \bigcap_{u' \in [0,t]} f^{-u'}(E') \right) \right\} \\
&\qquad= \left\{ (t,x) \in \R_{\geqslant 0}\times X \ \middle| \ x \in \bigcap_{u \in [0,a+b]} f^{-u}(E) \cap \bigcap_{u' \in [a, t+a+b+c]} f^{-u'}(E') \right\} 
\end{align*}
and
\begin{align*}
&\Dom F_{E''} = \left\{ (t,x) \in \R_{\geqslant 0}\times X \ \middle| \ x \in \bigcap_{u'' \in [0,t]} f^{-u''}(E'') \right\} \\
&\quad= \left\{ (t,x) \in \R_{\geqslant 0} \times X \ \middle| \ x \in \bigcap_{u \in [0,t+a+b]} f^{-u}(E) \cap \bigcap_{u' \in [a,t+a+b+c]} f^{-u'}(E') \right\}. 
\end{align*}
Since 
\begin{align*}
\bigcap_{u' \in [a, t+a+b+c]} f^{-u'}(E')
&= \bigcap_{v \in [a, t+a]} f^{-v} \left( \bigcap_{v' \in [0,b+c]} f^{-v'}(E') \right) \\
&\subset \bigcap_{v \in [a, t+a]} f^{-v}(f^{-b}(E))
= \bigcap_{u \in [a+b,t+a+b]} f^{-u}(E),
\end{align*}
we obtain
\[
\bigcap_{u' \in [a, t+a+b+c]} f^{-u'}(E') = \bigcap_{u \in [a+b,t+a+b]} f^{-u}(E) \cap \bigcap_{u' \in [a, t+a+b+c]} f^{-u'}(E')
\]
and
\begin{align*}
&\bigcap_{u \in [0,a+b]} f^{-u}(E) \cap \bigcap_{u' \in [a,t+a+b+c]} f^{-u'}(E') \\
&\qquad\qquad\qquad\qquad\qquad\qquad= \bigcap_{u \in [0,t+a+b]} f^{-u}(E) \cap \bigcap_{u' \in [a,t+a+b+c]} f^{-u'}(E').
\end{align*}
We have thus proved $\Dom F_{E''} = g^{-1}(\Dom F_{E'})$. 
\end{proof}

The proof of Theorem~\ref{thm:cont-isolating-2} is completed. 
\end{proof}

\begin{defn}
Let $F \colon \R_{\geqslant 0} \times X \partialto X$ 
be a continuous partial semiflow on a locally compact Hausdorff space $X$. 
Let $S$ be an isolated $F$-invariant subset of $X$. 
We define the \emph{Conley index of $S$ relative to $F$} to be 
$\Conl_F(E)$, where $E$ is an index neighbourhood of $S$. 
\end{defn}

By Theorems~\ref{thm:cont-isolating-1} and \ref{thm:cont-isolating-2}, 
the full subcategory of $\CptifSub_F$ consisting of all index neighbourhoods of $S$
is equivalent to the category with one object and one morphism. 
Thus, the Conley index $\Conl_F(E)$ is well-defined as an object of $\Sz_\cont(\CHaus_\ast)$ 
up to unique isomorphism. 

\begin{rmk}\label{rmk:cont-examples}
We can see in the same way as Example~\ref{ex:disc-examples} that 
there are $\sim_F$-equivalence classes of $F$-compactifiable subsets 
that cannot be obtained from 
Theorems~\ref{thm:cont-isolating-1} and \ref{thm:cont-isolating-2}.
\end{rmk}

\appendix

\part{Appendix}

\section{Proper maps of compactly generated weak Hausdorff spaces}\label{append:proper}

In this appendix, we prove some results on proper maps that are needed in this paper. 
As in the other parts of the paper, 
we assume every topological space to be compactly generated weak Hausdorff. 

We define proper maps as follows: 

\begin{defn}
Let $f \colon X \to Y$ be a continuous map of topological spaces. 
We say that $f$ is \emph{proper} if, for any compact Hausdorff subset $L$ of $Y$, 
the inverse image $f^{-1}(L)$ is compact Hausdorff. 
\end{defn}

\begin{lem}\label{lem:proper-composition-cancellation}
Let $f \colon X \to Y$ and $g \colon Y \to Z$ be two continuous maps. 
\begin{enumerate}
\item If $f$ and $g$ are proper, $gf$ is also proper. 
\item If $gf$ is proper, $f$ is also proper. 
\end{enumerate}
\end{lem}
\begin{proof}
(1) Obvious. 

(2) Let $L$ be a compact Hausdorff subset of $Y$. We have 
\[
f^{-1}(L) \subset (gf)^{-1}(g(L)). 
\]
Since $Z$ is weak Hausdorff, $g(L)$ is a compact Hausdorff subset of $Z$. 
Since $gf$ is proper, the inverse image $(gf)^{-1}(g(L))$ is a compact Hausdorff subset of $X$. 
On the other hand, since $Y$ is weak Hausdorff, $L$ is closed in $Y$. 
Hence, $f^{-1}(L)$ is closed in $X$ (and therefore, in $(gf)^{-1}(g(L))$). 
Thus, $f^{-1}(L)$ is compact Hausdorff. 
\end{proof}

\begin{lem}\label{lem:proper-base-change}
Proper maps are stable under base change. 
\end{lem}
\begin{proof}
Let $f \colon X \to Y$ be a proper map and $g \colon Y' \to Y$ be a continuous map. 
Consider the following pullback diagram: 
\[
\begin{tikzcd}
X \times_Y Y' \arrow[r, "g'"] \arrow[d, "f'" left] & X \arrow[d, "f"] \\
Y' \arrow[r, "g"] & Y.
\end{tikzcd}
\]
Take any compact Hausdorff subset $L'$ of $Y'$. 
Since $Y$ is weak Hausdorff, $g(L')$ is a compact Hausdorff subset of $Y$. 
Since $f$ is proper, the inverse image 
$f^{-1}(g(L')) \ (= g'(f'^{-1}(L')))$ is a compact Hausdorff subset of $X$. 
The following is a pullback diagram: 
\[
\begin{tikzcd}
f'^{-1}(L) \arrow[d, "f'" left] \arrow[r, "g'"] & f^{-1}(g(L')) \arrow[d, "f"] \\
L' \arrow[r, "g"] & g(L').
\end{tikzcd}
\]
We see that $f'^{-1}(L)$ is compact Hausdorff 
because it is a fibre product of two compact Hausdorff spaces. Thus, $f'$ is proper. 
\end{proof}

\begin{lem}\label{lem:proper-equivalent}
Let $f \colon X \to Y$ be a continuous map of topological spaces. 
Then, the following two conditions are equivalent: 
\begin{enumerate}[label = (\roman*)]
\item $f$ is proper. 
\item For any continuous map $g \colon L \to Y$ with $L$ compact Hausdorff, 
the fibre product $L \times_Y X$ is compact Hausdorff. 
\end{enumerate}
\end{lem}
\begin{proof}
(ii)~$\Rightarrow$~(i): Consider the case when $g$ is an inclusion. 

(i)~$\Rightarrow$~(ii): By Lemma~\ref{lem:proper-base-change}, 
the base change $f' \colon L \times_Y X \to L$ is proper. Since $L$ is compact Hausdorff, 
it follows that $L \times_Y X \ (= f'^{-1}(L))$ is also compact Hausdorff. 
\end{proof}

\begin{lem}\label{lem:proper-universally-closed}
Proper maps are universally closed. 
\end{lem}
\begin{proof}
By Lemma~\ref{lem:proper-base-change}, it suffices to verify that proper maps are closed. 
Let $f \colon X \to Y$ be a proper map. 
Let $A$ be a closed subset of $X$. Since $Y$ is compactly generated, 
it suffices to verify that, 
for any continuous map $g \colon L \to Y$ with $L$ compact Hausdorff, 
$g^{-1}(f(A))$ is closed in $L$. Consider the following pullback diagram: 
\[
\begin{tikzcd}
L \times_Y X \arrow[d, "g'" left] \arrow[r, "f'"] & L \arrow[d, "g"] \\
X \arrow[r, "f"] & Y.
\end{tikzcd}
\]
By Lemma~\ref{lem:proper-equivalent}, $L \times_Y X$ is compact Hausdorff. 
We see that $f' \colon L \times_Y X \to L$ is a continuous map between compact Hausdorff spaces, 
hence closed. Thus, $g^{-1}(f(A)) \ (= f'(g'^{-1}(A)))$ is closed in $L$. 
\end{proof}

\begin{lem}
An inclusion is proper if and only if it is closed. 
\end{lem}
\begin{proof}
The `if' part holds because closed subsets of compact Hausdorff spaces are compact Hausdorff. 
The `only if' part follows from Lemma~\ref{lem:proper-universally-closed}. 
\end{proof}

\begin{lem}\label{lem:proper-source-compact}
Let $f \colon K \to X$ be a continuous map with $K$ compact Hausdorff. 
Then, $f$ is proper. 
\end{lem}
\begin{proof}
Take any compact Hausdorff subset $L$ of $X$. Since $X$ is weak Hausdorff, $L$ is closed in $X$, 
and the inverse image $f^{-1}(L)$ is closed in $K$. 
Since $K$ is compact Hausdorff, $f^{-1}(L)$ is also compact Hausdorff. 
\end{proof}

\begin{lem}\label{lem:proper-finite-closed-cover}
Let $f \colon X \to Y$ be a continuous map of topological spaces. 
Suppose that there exists a finite closed cover $(A_i)_{i \in I}$ of $X$ such that 
$f|_{A_i} \colon A_i \to Y$ is proper for each $i \in I$. Then, $f$ is proper. 
\end{lem}
\begin{proof}
Take any compact Hausdorff subset $L$ of $Y$. 
For each $i \in I$, we see that $f^{-1}(L) \cap A_i$ is compact Hausdorff. We have a continuous map 
\[
g \colon \coprod_{i \in I} (f^{-1}(L) \cap A_i) \to X
\]
induced from the inclusions. 
Since $X$ is weak Hausdorff, $f^{-1}(L)$, which is the image of $g$, is compact Hausdorff.
\end{proof}

\section{Shift equivalences and the Szymczak category}\label{append:shift-equivalence}

In this appendix, we introduce a class of morphisms, called shift equivalences, 
and explain its relation to the Szymczak category, 
following the idea of Franks--Richeson~\cite{FR00}. 
The results in this appendix are not used in this paper. 
We include them, however, because they clarify a categorical meaning of the Szymczak category. 

We use the language of localization of categories. For basics on this subject, 
see e.g.\ \cite[\href{https://stacks.math.columbia.edu/tag/04VB}{\S 04VB}]{Stacks}.

\begin{defn}
Let $\C$ be a category. 
\begin{enumerate}
\item Let $f \colon X \to X$ be an endomorphism in $\C$. 
Then, we write $\widehat{f}$ for $f$ seen as a morphism from $f$ to itself in $\End(\C)$. 
\item Let $f \colon X \to X$ and $g \colon Y \to Y$ be two endomorphisms in $\C$. 
We say that a morphism $\varphi \colon f \to g$ in $\End(\C)$ is a \emph{shift equivalence} 
if there exist a morphism $\psi \colon g \to f$ in $\End(\C)$ and $n \in \N$ such that 
$\psi \varphi = \widehat{f}^n$ and $\varphi \psi = \widehat{g}^n$. 
\end{enumerate}
\end{defn}

We write $\ShiftEq(\C)$ for the class of all shift equivalences in 
the category $\End(\C)$. 

\begin{lem}\label{lem:disc-saturated-ms}
Let $\C$ be a category. 
Then, $\ShiftEq(\C)$ is a saturated multiplicative system in $\End(\C)$
(in the sense of \cite[\href{https://stacks.math.columbia.edu/tag/05Q8}{Defn.~05Q8}]{Stacks}). 
\end{lem}

\begin{proof}
Obviously, $\ShiftEq(\C)$ is closed under composition and contains the isomorphisms. 
Suppose that we are given a diagram
\[
\begin{tikzcd}
f \arrow[r, "\chi"] \arrow[d, "\varphi" left] & h \\
g & 
\end{tikzcd}
\]
in $\End(\C)$ such that $\varphi \in \ShiftEq(\C)$. 
Let us take $\psi \colon g \to f$ and $n \in \N$ such that 
$\psi\varphi = \widehat{f}^n$ and $\varphi\psi = \widehat{g}^n$. 
Then, the following diagram is commutative, and the right map 
$\widehat{h}^n$ is a morphism in $\ShiftEq(\C)$: 
\[
\begin{tikzcd}
f \arrow[r, "\chi"] \arrow[d, "\varphi" left] & h \arrow[d, "\widehat{h}^n"] \\
g \arrow[r, "\chi\psi"] & h.
\end{tikzcd}
\]
Suppose that $\varphi, \psi \colon f \rightrightarrows g$ 
are two morphisms in $\End(\C)$ and $\chi \colon h \to f$ is a morphism in $\ShiftEq(\C)$ such that 
$\varphi\chi = \psi\chi$. 
Let us take $\omega \colon f \to h$ and $n \in \N$ such that 
$\omega\chi = \widehat{h}^n$ and $\chi\omega = \widehat{f}^n$. 
Then, $\widehat{g}^n \varphi = \widehat{g}^n \psi$, 
and $\widehat{g}^n$ is a morphism in $\ShiftEq(\C)$. 
Thus, $\ShiftEq(\C)$ is a left multiplicative system. 
One can see that $\ShiftEq(\C)$ is a right multiplicative system in the same way. 

Let us prove that $\ShiftEq(\C)$ is saturated. Suppose that we are given a sequence 
\[
f \xrightarrow{\varphi} g \xrightarrow{\psi} h \xrightarrow{\chi} i 
\]
in $\End(\C)$ such that $\psi\varphi$ and $\chi\psi$ are morphisms in $\ShiftEq(\C)$. 
We want to prove that $\psi$ is in $\ShiftEq(\C)$. 
Take $\omega \colon h \to f$, $\tau \colon i \to g$, and $m,n \in \N$ such that 
\[
\omega\psi\varphi = \widehat{f}^m, \qquad \psi\varphi\omega = \widehat{h}^m, \qquad
\tau\chi\psi = \widehat{g}^n, \qquad \chi\psi\tau = \widehat{i}^n.
\]
Put $\sigma = \varphi \omega \widehat{h}^n$. Then, we have 
$\sigma \psi = \widehat{g}^{m+n}$ and $\psi \sigma = \widehat{h}^{m+n}$. 
\end{proof}

\begin{defn}
Let $\C$ be a category. 
We define a functor $Q \colon \End(\C) \to \Sz(\C)$ as follows: 
\begin{itemize}
\item $Q$ is identity on the objects. 
\item $Q\varphi = \overline{(\varphi, 0)}$ for each morphism $\varphi$ in $\End(\C)$. 
\end{itemize}
\end{defn}

\begin{prop}\label{prop:disc-localization}
Let $\C$ be a category. 
\begin{enumerate}
\item $(\Sz(\C), Q)$ is a localization of the category $\End(\C)$ 
at the class of morphisms $\ShiftEq(\C)$. 
\item A morphism $\varphi$ in $\End(\C)$ is a shift equivalence if and only if 
$Q \varphi$ is an isomorphism. 
\end{enumerate}
\end{prop}

\begin{proof}
(1) Let $f$ and $g$ be two endomorphisms in $\C$.
Let $\varphi \colon f \to g$ be a shift equivalence. 
Take $\psi \colon g \to f$ and $n \in \N$ such that $\psi\varphi = \widehat{f}^n$ and 
$\varphi\psi = \widehat{g}^n$. 
We see that $Q\varphi \ (= \overline{(\varphi, 0)})$ is an isomorphism in $\Sz(\C)$ 
with inverse $\overline{(\psi, n)}$. 

Let $\Phi \colon \End(\C) \to \D$ be a functor that sends shift equivalences 
to isomorphisms in $\D$. 
We want to show that there exists a unique functor 
$\overline{\Phi} \colon \Sz(\C) \to \D$ 
such that $\overline{\Phi}Q = \Phi$. 
Obviously, such $\overline{\Phi}$ must be equal to $\Phi$ on the objects. 
Notice that, 
for any morphism $\varphi \colon f \to g$ in $\End(\C)$ and any $n \in \N$, we have 
$\overline{(\varphi, n)} = Q\varphi \circ (Q \widehat{f})^{-n}$. 
This means that $\overline{\Phi}$ must satisfy 
$\overline{\Phi}\,\overline{(\varphi, n)} = \Phi \varphi \circ (\Phi \widehat{f})^{-n}$. 
To the contrary, if we define $\overline{\Phi}$ in this way, 
it is indeed a well-defined functor satisfying $\overline{\Phi}Q = \Phi$. 

(2) This follows from (1), Lemma~\ref{lem:disc-saturated-ms}, and 
\cite[\href{https://stacks.math.columbia.edu/tag/05Q9}{Lem.~05Q9}]{Stacks}.
\end{proof}

Next, let us consider the continuous time case. 

\begin{defn}
Let $\C$ be either $\Set_\ast$ or $\CHaus_\ast$. 
\begin{enumerate}
\item Let $F$ be an object of $\SFlow(\C)$. 
Then, for each $t \in \R_{\geqslant 0}$, 
we write $\widehat{f}^t$ for $f^t$ seen as a morphism from $F$ to itself in $\SFlow(\C)$. 
\item Let $F$ and $G$ be two objects in $\SFlow(\C)$.
We say that a morphism $\varphi \colon F \to G$ in $\SFlow(\C)$ 
is a \emph{shift equivalence} if there exist a morphism 
$\psi \colon G \to F$ in $\SFlow(\C)$ and $s \in \R_{\geqslant 0}$ such that 
$\psi \varphi = \widehat{f}^s$ and $\varphi \psi = \widehat{g}^s$. 
\end{enumerate}
\end{defn}

We write $\ShiftEq_\cont(\C)$ for the class of all shift equivalences in 
the category $\SFlow(\C)$. 

\begin{lem}\label{lem:cont-saturated-ms}
Let $\C$ be either $\Set_\ast$ or $\CHaus_\ast$. 
Then, $\ShiftEq_\cont(\C)$ is a saturated multiplicative system in $\SFlow(\C)$. 
\end{lem}

\begin{proof}
The same as the proof of Lemma~\ref{lem:disc-saturated-ms}. 
\end{proof}

\begin{defn}
Let $\C$ be either $\Set_\ast$ or $\CHaus_\ast$. 
We define a functor $Q \colon \SFlow(\C) \to \Sz_\cont(\C)$ as follows: 
\begin{itemize}
\item $Q$ is identity on the objects. 
\item $Q\varphi = \overline{(\varphi, 0)}$ for each morphism $\varphi$ in $\SFlow(\C)$. 
\end{itemize}
\end{defn}

\begin{prop}\label{prop:cont-localization}
Let $\C$ be either $\Set_\ast$ or $\CHaus_\ast$. 
\begin{enumerate}
\item $(\Sz_\cont(\C), Q)$ is a localization of the category $\SFlow(\C)$ 
at the class of morphisms $\ShiftEq_\cont(\C)$. 
\item A morphism $\varphi$ in $\SFlow(\C)$ is a shift equivalence if and only if 
$Q \varphi$ is an isomorphism. 
\end{enumerate}
\end{prop}

\begin{proof}
The same as the proof of Proposition~\ref{prop:disc-localization}. 
\end{proof}

\begin{rmk}
The relation between the Szymczak category and shift equivalences are explained in 
Franks--Richeson~\cite{FR00}. However, in \cite{FR00}, localization is not used; 
the Conley index is thus defined as a shift equivalence class in the endomorphism category. 
In other words, it is defined as an \emph{isomorphism class} 
of the objects in the Szymczak category. 
The Conley index defined in Szymczak~\cite{Szy95} carries more information: 
it is a connected simple system in the Szymczak category, 
i.e.\ an object of the Szymczak category defined up to \emph{unique isomorphism}, 
which retains the information of automorphisms. 
\end{rmk}

\subsection*{Acknowledgements}
I would like to thank Hokuto Konno for his helpful comments on Floer theory. 
This work was supported by JSPS KAKENHI Grant Number 19K14529.

\end{document}